\documentclass[oneside,12pt]{amsart}

\usepackage[T1]{fontenc}
\usepackage[utf8]{inputenc}

\usepackage{amsmath,amssymb,amsfonts,amsthm,mathtools}
\usepackage{mathrsfs}

\usepackage{url}
\usepackage{hyperref}

\theoremstyle{plain}
\newtheorem{theorem}{Theorem}[section]
\newtheorem{proposition}[theorem]{Proposition}
\newtheorem{lemma}[theorem]{Lemma}
\newtheorem{corollary}[theorem]{Corollary}

\theoremstyle{definition}
\newtheorem{definition}[theorem]{Definition}

\theoremstyle{remark}
\newtheorem{remark}[theorem]{Remark}
\newtheorem{example}[theorem]{Example}

\DeclareMathOperator{\Endop}{End}
\newcommand{\End}[1]{\Endop\!\left(#1\right)}
\newcommand{\Endp}[1]{\Endop_{+}\!\left(#1\right)}
\newcommand{\Endm}[1]{\Endop_{-}\!\left(#1\right)}

\newcommand{\Ann}{\mathrm{Ann}}
\newcommand{\GL}{\mathrm{GL}}
\newcommand{\SO}{\mathrm{SO}}

\newcommand{\SU}{\mathrm{SU}}

\newcommand{\Sp}{\mathrm{Sp}}

\newcommand{\bbR}{\mathbb{R}}
\newcommand{\bbC}{\mathbb{C}}

\newcommand{\MM}{\mathbf{M}}
\newcommand{\N}{\mathbb{N}}

\newcommand{\frakC}{\mathfrak{C}}

\newcommand{\frakg}{\mathfrak{g}}

\newcommand{\Iso}{\mathrm{I}}

\newcommand{\trace}{\mathrm{tr}}

\newcommand{\bbH}{\mathbb{H}}
\newcommand{\Kern}{\mathrm{Kern}}

\newcommand{\ad}{\mathrm{ad}}
\newcommand{\Ad}{\mathrm{Ad}}

\newcommand{\rmG}{\mathrm{G}}

\newcommand{\rmS}{\mathrm{S}}

\newcommand{\bbE}{\mathbb{E}}

\newcommand{\scrA}{\mathscr{A}}
\newcommand{\scrF}{\mathscr{F}}

\newcommand{\scrK}{\mathcal{K}}

\newcommand{\scrM}{\mathcal{M}}

\newcommand{\scrR}{\mathcal{R}}

\newcommand{\scrP}{\mathscr{P}}

\newcommand{\Ric}{\mathrm{Ric}}

\renewcommand{\d}{\mathrm{d}}

\newcommand{\scal}{\mathrm{s}}

\newcommand{\rmR}{\mathrm{R}}

\newcommand{\Eval}{\mathrm{Eval}}

\newcommand{\Spin}{\mathbf{Spin}}

\renewcommand{\min}{\mathrm{min}}
\newcommand{\C}{\mathrm{C}}

\newcommand{\symR}{R}

\title{The Minimal Polynomial of a Riemannian \(\frakC_0\)\,-space}
\author{Tillmann Jentsch}
\address{Unidad Cuernavaca del Instituto de Matem\'aticas, Universidad
Nacional Aut\'onoma de M\'exico, Avenida Universidad s/n, Lomas de
Chamilpa, 62210 Cuernavaca, Morelos, MEXICO (Guest visitor)}
\email{tilljentsch@gmail.com}


\begin{document}\sloppy
\begin{abstract}
We construct, at each point of a Riemannian \(\frakC_0\)-space, a
polynomial in one variable whose coefficients are polynomial functions
on the tangent space. For a real-analytic Riemannian \(\frakC_0\)-space
these pointwise-defined polynomials glue together to a global polynomial
whose coefficients are Killing tensors invariant under the full isometry group.
Moreover, the degree of this polynomial provides an upper bound for the Singer
invariant of the space.
\end{abstract}

\maketitle
\section{Overview}
\label{se:overview}

Let \(M\) be a smooth manifold equipped with a nondegenerate
symmetric tensor field \(g\) of type \((0,2)\). The pair \((M,g)\) is
called a pseudo-Riemannian (or semi-Riemannian) manifold. For brevity
we will refer to \((M,g)\) as \emph{Riemannian} (allowing indefinite signature);
whenever a positive-definite metric is needed, we will state this explicitly.

A Riemannian manifold \((M,g)\) carries a unique torsion-free metric
connection \(\nabla\), the Levi–Civita connection. Its curvature tensor is defined by
\[
  \rmR(X,Y)Z \;\coloneqq\; \nabla_X \nabla_Y Z \;-\; \nabla_Y \nabla_X Z
  \;-\; \nabla_{[X,Y]} Z
\]
for all \(X,Y,Z \in \Gamma(TM)\), i.e., sections of
the tangent bundle (or simply vector fields).

For \(k \ge 0\), let \(\nabla^k \rmR\) denote the \(k\)th iterated
covariant derivative of \(\rmR\), defined recursively by
\[
  \nabla^0 \rmR \coloneqq \rmR,\qquad
  \nabla^k \rmR \coloneqq
  \nabla(\nabla^{k-1} \rmR)
  \quad (k \ge 1)
\]
via the induced covariant derivative \(\nabla\) acting on tensor fields
in the standard way. For instance, locally symmetric spaces are
characterized by \(\nabla \rmR = 0\). However, in the positive-definite
case, the condition \(\nabla^k \rmR = 0\) for some \(k \ge 2\) does not
yield new examples of compact Riemannian manifolds (and, more generally,
of complete ones; see~\cite{NO}) beyond locally symmetric spaces.

Therefore, motivated by the jet isomorphism theorem (see~\cite{J1}),
we consider instead the symmetrized covariant derivatives of \(\rmR\):
Given \(X \in T_pM\), define an endomorphism
\(\symR^k(X) \in \End{T_pM}\) by
\begin{equation}\label{eq:kth_order_symmetrized_curvature_tensor_1}
  \symR^k(X)Y \coloneqq
  (\nabla^{k}_{X,\ldots,X}\,\rmR)(Y,X)X
  \qquad (Y \in T_pM)
\end{equation}
Then the assignment
\begin{equation}\label{eq:kth_order_symmetrized_curvature_tensor_2}
  \symR^k|_p \colon T_pM \longrightarrow \End{T_pM},\qquad
  X \longmapsto \symR^k|_p(X)
\end{equation}
defines a homogeneous polynomial map of degree \(k+2\) from \(T_pM\) to
\(\End{T_pM}\). We refer to \(\symR^k\) as the
\emph{symmetrized \(k\)th covariant derivative of the curvature tensor}.

Let
\begin{equation}\label{eq:End_plus}
  \Endp{T_pM} \coloneqq
  \{\,A\in\End{T_pM}\mid A^{*}=A\,\}
\end{equation}
denote the space of \(g_p\)-self-adjoint endomorphisms. Then
\(\symR^k|_p(X)\in \Endp{T_pM}\) for all \(k \ge 0\). Indeed, for
\(Y,Z \in T_pM\), pair symmetry (which is preserved under covariant
differentiation) gives
\begin{equation}\label{eq:scrR_is_self_adjoint}
\begin{aligned}
  g_p(\symR^k(X)Y,Z)
  &= (\nabla^{k}_{X,\ldots,X}\,\rmR)(Y,X,X,Z) \\
  &= (\nabla^{k}_{X,\ldots,X}\,\rmR)(Z,X,X,Y) \\
  &= g_p(Y,\symR^k(X)Z)
\end{aligned}
\end{equation}
However, the self-adjointness of \(\symR^k|_p(X)\) will be used only to
prove Theorem~\ref{th:main_addition} and the subsequent corollary.
Otherwise it suffices to work with the larger space \(\End{T_pM}\).

\begin{remark}\label{re:additional_properties}
Besides the fact that \(\symR^k|_p(X)\) is \(g_p\)-self-adjoint for all
\(X \in T_pM\), the symmetrized \(k\)th covariant derivative of the
curvature tensor has additional algebraic properties. Most prominently,
\(\symR^k|_p\) belongs to a distinguished subspace of the tensor space
\(\bigotimes^{k+4} T_pM^*\) corresponding to the Young diagram of shape
\((k+2,2)\); cf.\ \cite[\S2.1]{J1}. However, this will not be used
explicitly at any point in the sequel.
\end{remark}

Let \(\scrP^j(T_pM)\) be the vector space of homogeneous polynomials of
degree \(j\) on \(T_pM\), and let \(\scrP^j\bigl(T_pM;\End{T_pM}\bigr)\)
denote the vector space of homogeneous polynomials of degree \(j\) on \(T_pM\)
with values in \(\End{T_pM}\). As noted above,
\[
  \symR^k|_p \in \scrP^{k+2}(T_pM;\End{T_pM})
\]

For every integer \(k \ge 0\), the {\em symmetrized \(k\)-jet},
\(\symR^{\leq k}|_p \coloneqq (\symR^0|_p,\symR^1|_p,\ldots,\symR^k|_p)\)
of the curvature tensor belongs to the graded vector space
\[
  \scrP(T_pM;\End{T_pM}) = \bigoplus_{j=0}^\infty
  \scrP^j(T_pM;\End{T_pM})
\]
of \(\End{T_pM}\)-valued polynomial functions on \(T_pM\).
This space is a graded module over the graded \(\bbR\)-algebra
\(\scrP(T_pM) = \bigoplus_{i=0}^\infty \scrP^i(T_pM)\) of
real-valued polynomial functions on \(T_pM\).

Let
\begin{equation}\label{eq:eval_1}
  \Eval_p \colon \scrP(T_pM)[\lambda]
  \longrightarrow \scrP(T_pM;\End{T_pM})
\end{equation}
be the unique homomorphism of \(\scrP(T_pM)\)-modules such that
\begin{equation}\label{eq:eval_2}
  \Eval_p(\lambda^i) \coloneqq \symR^i|_p
  \qquad \text{for all } i \ge 0
\end{equation}
If
\[
  P = \sum_{i=0}^k a_i\,\lambda^{k-i}
  \in \scrP(T_pM)[\lambda]
\]
is a nonzero polynomial of degree \(k\), then
\begin{equation*}\
  \Eval_p(P)
  =
  \sum_{i=0}^k a_i\,\symR^{k-i}|_p
\end{equation*}
which lies in \(\scrP(T_pM;\End{T_pM})\).

Let \(P = \sum_{i=0}^k a_i\,\lambda^{k-i} \in \scrP(T_pM)[\lambda]\)
be a polynomial of degree \(k\) in \(\lambda\) with coefficients
\(a_i \in \scrP(T_pM)\). We say that \(P\) is \emph{homogeneous}
if \(a_i \in \scrP^i(T_pM)\) for \(i=0,\ldots,k\),
and \emph{monic} if \(a_0 = 1\). In particular,
\[
 \Eval_p(P) \in \scrP^{k+2}(T_pM;\End{T_pM})
\]

\begin{definition}\label{de:admissible_in_p}
Let \(p \in M\). A monic homogeneous polynomial
\(P = \sum_{i=0}^k a_i\,\lambda^{k-i} \in \scrP(T_pM)[\lambda]\)
is called \emph{admissible} if \(\Eval_p(P) = 0\). Since \(a_0=1\),
this means
\begin{equation}\label{eq:admissible_in_p}
\symR^k|_p = -\sum_{i=1}^k a_i\,\symR^{k-i}|_p
\end{equation}
\end{definition}

\begin{remark}
  In~\cite{JW}, we called~\eqref{eq:admissible_in_p} a “Jacobi relation” because it
  simultaneously encodes linear relations among the covariant
  derivatives of the Jacobi operators along geodesics emanating from
  \(p\).
\end{remark}

The fact that the polynomial ring \(\scrP(T_pM)\) is Noetherian implies that
an admissible polynomial exists at every point of an arbitrary Riemannian manifold;
cf. Lemma~\ref{le:admissible}. Hence, we need to refine the definition.

\begin{definition}\label{de:minimal_polynomial}
Let \(p\in M\) and suppose that a monic homogeneous polynomial
\(P = \sum_{i=0}^k a_i\,\lambda^{k-i}\in\scrP(T_pM)[\lambda]\)
is admissible.
We call \(P\) \emph{minimal} if the set
\begin{equation}\label{eq:U(p)_1}
   U(p)\coloneqq\bigl\{X\in T_pM \,\big|\, \{\symR^j(X)\}_{j=0}^{k-1}
  \text{ is linearly independent in } \End{T_pM}\bigr\}
\end{equation}
is nonempty.
\end{definition}

This set is Zariski open in \(T_pM\); hence, if it is nonempty, it is dense in the
Euclidean topology of \(T_pM\). Consequently, the degree \(k\) and the coefficients
\(a_i\) are determined by the symmetrized \(k\)-jet \(\symR^{\le k}\big|_p \coloneqq
\bigl(\symR^0\big|_p,\ldots,\symR^k\big|_p\bigr)\) of the curvature tensor and are
therefore intrinsic invariants of the \((k+2)\)-jet
of the metric tensor at \(p\). Hence, if a minimal polynomial exists, it is unique.
We denote it by \(P_{\min}(M,g,p)\) and refer to it as the minimal polynomial of \((M,g)\) at \(p\).
Its existence, however, is a nontrivial question and is not \emph{a priori} guaranteed;
indeed, for a generic metric tensor it fails to exist.

\DeclareRobustCommand{\Czero}{\texorpdfstring{\(\frakC_0\)}{C0}}
\subsection{The minimal polynomial of a Riemannian \Czero-space}

Let \(I \subseteq \bbR\) be an open interval of the real numbers, and let
\(\gamma\colon I \to M\) be a geodesic of \((M,g)\).
The Jacobi operator \(\symR_\gamma\) is the endomorphism-valued tensor field along
\(\gamma\) defined by
\begin{equation}\label{eq:Jacobi_operator}
  \symR_\gamma(t) \coloneqq \rmR\bigl(\,\cdot\,,\dot\gamma(t)\bigr)\dot\gamma(t)\colon T_{\gamma(t)}M
  \longrightarrow T_{\gamma(t)}M,\qquad
  X \longmapsto \rmR\bigl(X,\dot\gamma(t)\bigr)\dot\gamma(t)
\end{equation}
cf.\ \cite[Ch.~2.8]{BTV}. We define its first covariant derivative along
\(\gamma\) by
\[
  \symR^{(1)}_\gamma(t) \coloneqq \frac{\nabla}{\d t}\,\symR_\gamma(t)
\]
that is, for every vector field \(Y(t)\) along \(\gamma\),
\begin{equation}\label{eq:Jacobi_operator_abgeleitet}
  \bigl(\symR^{(1)}_\gamma(t)\bigr)Y(t)
  \coloneqq
  \frac{\nabla}{\d t}\bigl(\symR_\gamma(t)Y(t)\bigr)
  -
  \symR_\gamma(t)\,\frac{\nabla}{\d t}Y(t)
\end{equation}
Along a geodesic this is equivalently
\[
  \symR^{(1)}_\gamma(t)
  = (\nabla_{\dot\gamma(t)}\rmR)(\,\cdot\,,\dot\gamma(t))\,\dot\gamma(t)
\]
since \(\frac{\nabla}{\d t}\,\dot\gamma(t) \equiv 0\).

\begin{definition}[{\cite[Ch.~2.9]{BTV}}]\label{de:C_0}
A Riemannian manifold \((M,g)\) is called a \(\frakC_0\)-space if,
for every geodesic \(\gamma\colon I \to M\), there exists a
parallel skew-symmetric endomorphism field \(\C_\gamma\) along \(\gamma\)
such that
\begin{equation}\label{eq:def_C_0}
  \symR^{(1)}_\gamma(t) = [\C_\gamma(t),\,\symR_\gamma(t)]
  \quad\text{for all } t \in I.
\end{equation}
\end{definition}
Here,
\[
  [\C_\gamma(t),\,\symR_\gamma(t)] \coloneqq \C_\gamma(t)\circ\symR_\gamma(t)
  - \symR_\gamma(t)\circ\C_\gamma(t)
\] denotes the usual commutator of endomorphisms.
Also \(\C_\gamma\) being parallel means
\[
  \tfrac{\nabla}{\d t}\C_\gamma(t) \equiv 0
\]

For instance, every G.O.\ space is a \(\frakC_0\)-space.
Recall that a Riemannian manifold \((M,g)\) is called a G.O.\ space if every geodesic is an orbit
of a one-parameter subgroup of \(\Iso(M,g)\) (see \cite[Ch.~2.2]{BTV}).

The symmetrized jet
\((\symR^0|_p,\symR^1|_p,\ldots)\) of the curvature tensor on a
\(\frakC_0\)-space has a decisive algebraic property: the kernel
\(\Kern(\Eval_p)\) of the evaluation map defined by~\eqref{eq:eval_1}
and~\eqref{eq:eval_2} is a principal ideal of \(\scrP(T_pM)[\lambda]\); see
Corollary~\ref{co:C_0_ideal}.\footnote{This is all the more surprising
  because \(\scrP(T_pM)[\lambda]\) is not a principal ideal domain.}
This provides the algebraic half of the proof of the following theorem.

\begin{theorem}\label{th:main}
  Let \((M,g)\) be a \(\frakC_0\)-space. Then the minimal polynomial
  \(P_{\min}(M,g,p)\) exists at every \(p \in M\) and generates the
  kernel of the evaluation map~\eqref{eq:eval_2} as an ideal of
  \(\scrP(T_pM)[\lambda]\).
\end{theorem}

To conclude the proof of Theorem~\ref{th:main}, it remains to show that the unique monic generator of
\(\Kern(\Eval_p)\) satisfies Definition~\ref{de:minimal_polynomial}.
For this, after extending scalars from \(\scrP(T_pM)\) to the field of
rational functions \(\scrF(T_pM)\), one can construct, at each point of
any Riemannian manifold, a monic homogeneous polynomial
\[
  \widetilde P_{\min}(M,g,p) \in \scrF(T_pM)[\lambda]
\]
satisfying all requirements of
Definition~\ref{de:minimal_polynomial}, except that its coefficients
need not lie in \(\scrP(T_pM)\); for a generic Riemannian metric, they
do not. Comparing the two constructions shows that these coefficients
are in fact polynomial and that the resulting monic polynomial
generates \(\Kern(\Eval_p)\).

The minimal polynomial \(P_{\min}(M,g,p)\) of a \(\frakC_0\)-space is hence
uniquely characterized as the monic polynomial of minimal degree with
coefficients in \(\scrP(T_pM)\) such that \(\Eval_p\bigl(P_{\min}(M,g,p)\bigr) = 0\).
Moreover, if \(\Eval_p(Q) = 0\) for some
\(Q \in \scrP(T_pM)[\lambda]\), then \(P_{\min}(M,g,p)\) divides \(Q\) which
justifies the terminology from an algebraic point of view.
In particular, a monic homogeneous \(\widetilde P \in \scrP(T_pM)[\lambda]\)
is admissible if and only if it can be written as
\[
  \widetilde P = Q\,P_{\min}(M,g,p)
\]
for a unique \(Q \in \scrP(T_pM)[\lambda]\).
Because \(P_{\min}(M,g,p)\) and \(\widetilde P\) are monic and homogeneous, so is \(Q\).
In this way admissible polynomials correspond bijectively to monic
homogeneous polynomials.

\subsection{Furher properties of the minimal polynomial of a Riemannian
\(\frakC_0\)-space}
Let \((M,g)\) be a \(\frakC_0\)-space, \(p \in M\) and let
\(P_{\min}(M,g,p) = \sum_{i=0}^k a_i\,\lambda^{k-i}\) be the minimal
polynomial. View \(P_{\min}(M,g,p)\) as a map
\[
  P_{\min}(M,g,p)\colon T_pM \longrightarrow \bbR[\lambda],\qquad
  X \longmapsto P_{\min}(M,g,p)(X) \coloneqq
  \sum_{i=0}^k a_i(X)\,\lambda^{k-i}
\]
Each value \(P_{\min}(M,g,p)(X)\) is a polynomial over \(\bbR\) of the
same degree \(k\), called the \emph{specialization} of \(P_{\min}(M,g,p)\)
at \(X\), i.e., the polynomial in \(\lambda\) obtained by evaluating the
coefficient functions \(a_i\) at \(X\).

\begin{theorem}\label{th:main_addition}
  Let \((M,g)\) be a \(\frakC_0\)-space, \(p \in M\), and let
  \(P_{\min}(M,g,p) = \sum_{i=0}^k a_i\,\lambda^{k-i}\) be the minimal
  polynomial.
  \begin{enumerate}
    \item If \(g\) is positive definite, then all complex roots of the
      specialization \(P_{\min}(M,g,p)(X) \in \bbR[\lambda]\) are purely
      imaginary and simple for every \(X \in U(p)\). 
    \item If \(\Ric|_p \neq 0\), then \(a_k = 0\).
  \end{enumerate}
\end{theorem}

\begin{corollary}\label{co:main_addition}
Let \((M,g)\) be a \(\frakC_0\)-space with positive definite metric tensor and let \(p \in M\).
Then the minimal polynomial is determined entirely by its even coefficients:
\begin{equation}\label{eq:min_pol_pos_def}
  P_{\min}(M,g,p) \;\;=\;\; \sum_{i=0}^{\lfloor k/2 \rfloor} a_{2i}\,\lambda^{\,k-2i}
\end{equation}
where \(a_{2i} \ge 0\) and \(a_{2i}\big|_{U(p)} > 0\).
If moreover \(\Ric|_p \neq 0\), then \(k\) is odd.
\end{corollary}

The proofs of Theorem~\ref{th:main_addition} and
Corollary~\ref{co:main_addition} are given in
Section~\ref{se:proof_of_theorem_main_addition}.

\subsection{Global results}
The concept of admissibility extends naturally to the global setting by considering
smooth sections of the bundle \(\scrP(M) = \bigoplus_{i=0}^\infty
\scrP^i(M) \to M\), whose fiber at \(p \in M\) is the graded algebra
\(\scrP(T_pM)\).
We refer to such sections as \emph{polynomial tensors}; via polarization, they correspond
bijectively to symmetric tensors.

Consider the polynomial algebra \(\Gamma\bigl(\scrP(M)\bigr)[\lambda]\)
in one variable \(\lambda\) over the graded algebra \(\Gamma\bigl(\scrP(M)\bigr)\).
We say that \(P \in \Gamma\bigl(\scrP(M)\bigr)[\lambda]\)
is \emph{monic homogeneous of degree \(k\)} if it can be written
\[
  P = \sum_{i=0}^k a_i\,\lambda^{k-i}, \qquad
  a_i \in \Gamma\bigl(\scrP^i(M)\bigr),\ \ a_0 = 1
\]
Then for each \(p \in M\),
\[
  P|_p = \sum_{i=0}^k a_i|_p\,\lambda^{k-i} \in \scrP(T_pM)[\lambda]
\]
is monic and homogeneous of degree \(k\).

\begin{definition}\label{de:admissible}
Let \((M,g)\) be a Riemannian manifold and \(P \in \Gamma\bigl(\scrP(M)\bigr)[\lambda]\)
monic homogeneous.
We call \(P\) \emph{admissible} if \(P|_p\) is admissible for each \(p \in M\).
\end{definition}

In Section~\ref{se:JR} we shall prove the following existence result.

\begin{theorem}\label{th:globally_admissible}
If \((M,g)\) is compact and real-analytic,
then there exists \(P \in \Gamma\bigl(\scrP(M)\bigr)[\lambda]\) that is admissible.
\end{theorem}
The proof uses the fact that, like the polynomial ring,
the ring of convergent power series is also Noetherian~\cite{ZS}.
We emphasize, however, that the mere existence of an admissible polynomial imposes few
geometric restrictions.

By contrast, an admissible \(P \in \Gamma\bigl(\scrP(M)\bigr)[\lambda]\) that is minimal at
every point is unique (if it exists). In this case we write \(P_{\min}(M,g) \coloneqq P\)
and call it the minimal polynomial of \((M,g)\). For a homogeneous Riemannian space
the pointwise and global notions coincide: if \(P_{\min}(M,g,p)\) exists at some \(p \in M\),
then it extends uniquely to \(P_{\min}(M,g) \in \Gamma\bigl(\scrP(M)\bigr)[\lambda]\) by
the action of the isometry group \(\Iso(M,g)\).

For a \(\frakC_0\)-space \((M,g)\) set
\[
  U(M) \coloneqq \{\, p \in M \mid \deg\bigl(P_{\min}(M,g,p)\bigr) \text{ attains its maximal value} \,\}
\]
Let \(m \coloneqq \dim(M)\).
Since \(P_{\min}(M,g,p)\) exists at each point of \(M\) by Theorem~\ref{th:main}, and its
degree \(k\) is bounded by \(m^2\), the dimension of \(\End{T_pM}\)\footnote{In fact, the degree is bounded by
\(\tfrac{m(m + 1)}{2}\), the dimension of the space of symmetric endomorphisms,
cf.~Remark~\ref{re:additional_properties}.}, according to
\eqref{eq:U(p)_1}, this set is nonempty.
It is also open and \(\Iso(M,g)\)-invariant; see again~\eqref{eq:U(p)_1}.
In particular, \(U(M) = M\) if \((M,g)\) is homogeneous.

Recall that a polynomial tensor \(a\) (more precisely, the corresponding symmetric tensor)
is called a \emph{Killing tensor} if \(a(\dot\gamma)\) remains
constant along every geodesic \(\gamma\) of \((M,g)\) (cf.\ \cite{HMS}).

\begin{theorem}\label{th:main_global}
Let \((M,g)\) be a \(\frakC_0\)-space, and let \(U(M) \subseteq M\) be the
set of points where \(\deg(P_{\min}(M,g,p))\) attains its maximal value.
Then the following hold:
\begin{enumerate}
  \item The minimal polynomial \(P_{\min}(U(M),g|_{U(M)})\)
    exists, and its coefficients \(a_i \in \Gamma(\scrP^i(U(M)))\)
    are \(\Iso(M,g)\)-invariant Killing tensors of degree \(i\) on \((U(M),\,g|_{U(M)})\).
  \item If \((M,g)\) is real-analytic, then the coefficients \(a_i\) extend to \(\Iso(M,g)\)-invariant
    Killing tensors on \((M,g)\).
\end{enumerate}
\end{theorem}
In particular, on any homogeneous \(\frakC_0\)-space \((M,g)\) (for example, a G.O.\ space),
each coefficient of the minimal polynomial \(P_{\min}(M,g)\) is a Killing tensor
invariant under the full isometry group. Recall that on a G.O. space each
\(\Iso(M,g)\)-invariant symmetric tensor is automatically Killling; however on a proper
\(\frakC_0\)-space (which is not G.O.) this is no longer true. Recall also that every
\(\frakC_0\)-space with positive definite metric is real-analytic (see \cite[Cor.\ 3]{BV}).

The proof of Theorem~\ref{th:main_global} is given in Section~\ref{se:further_properties}.
\subsection{Examples}
Combining the previous results, we see that every homogeneous \(\frakC_0\)-space
(including every G.O.\ space and hence, in particular, every naturally reductive space)
with a positive definite metric tensor carries a nontrivial collection of intrinsically defined,
\(\Iso(M,g)\)-invariant Killing tensors, namely the even coefficients \(a_{2i}\) of the
minimal polynomial \(P_{\min}(M,g)\), which satisfy the pointwise positivity
assertions of Corollary~\ref{co:main_addition}.

At first glance, this may seem unremarkable: the \(a_{2i}\) could simply be multiples of powers of
\(\|\,\cdot\,\|^2\), the squared norm induced by the metric tensor, \(a_{2i}(X) = c_i \|X\|^{2i}\)
with \(c_i \in \bbR_{>0}\). Indeed, in~\cite{JW} we showed that several compact homogeneous spaces
have minimal polynomials with exactly this property. The quaternionic Hopf-bundle case in Part~(b)
below, however, was not included in~\cite{JW}.

\begin{example}\label{ex:LJR}
On the following naturally reductive spaces, the even coefficients of the minimal polynomials are
real multiples of powers of \(\|\,\cdot\,\|^2\):
\begin{enumerate}
  \item On the \((2n+1)\)-dimensional total space \(\hat\MM^n(\kappa,\,r)\) of
      the Hopf fibration (with circle fibers of radius \(r\)) over a complex
      space form \(\MM^n(\kappa)\) of holomorphic sectional curvature
      \(\kappa \neq 0\), equipped with the Berger metric, the minimal
      polynomial is
      \begin{equation}\label{eq:LJR_BS}
        \lambda^3 \;+\; c^2\,\|\,\cdot\,\|^2\,\lambda
      \end{equation}
      where \(c^2 \coloneqq \tfrac14\,r^2\,\kappa^2\).
      Moreover, for each \(c > 0\), this same polynomial \eqref{eq:LJR_BS}
      is the minimal polynomial of the \((2n+1)\)-dimensional Heisenberg
      group \(\hat\MM^n(0,\,1/c^2)\), equipped with the multiple
      \(\tfrac{1}{c^2}g\) of the canonical left-invariant metric \(g\)
      of type~H.
   \item On the \((4n+3)\)-dimensional round sphere of radius \(r\),
      viewed as the total space of the quaternionic Hopf fibration
      \[
        \rmS^3(r)\longrightarrow
        \bigl(\rmS^{4n+3}(r),g_{\mathrm{round}}\bigr)
        \longrightarrow
        \bigl(\bbH P^n,g_{\mathrm{qFS}}(4/r^2)\bigr),
      \]
      where \(g_{\mathrm{round}}\) has constant sectional curvature
      \(1/r^2\) and \(g_{\mathrm{qFS}}(4/r^2)\) has constant
      quaternionic sectional curvature \(4/r^2\), let
      \[
        g_\tau=g_{\mathcal H}+\tau g_{\mathcal V},
        \qquad
        g_{\mathrm{round}}=g_{\mathcal H}+g_{\mathcal V},
      \]
      be the associated canonical variation. At \(\tau=\frac15\),
      the minimal polynomial of
      \(\bigl(\rmS^{4n+3}(r),g_{1/5}\bigr)\) is
      \[
        \lambda^3+
        \frac{4}{5r^2}
        \|\,\cdot\,\|_{g_{1/5}}^2\lambda.
      \]
    \item On a \(7\)-dimensional homogeneous nearly parallel \(\rmG_2\)-space of scalar curvature
        \(\scal\), the minimal polynomial is
        \begin{equation}\label{eq:LJR_NP}
          \lambda^3 \;+\; \tfrac{2\,\scal}{189}\,\|\,\cdot\,\|^2\,\lambda
        \end{equation}
  \item On a \(6\)-dimensional homogeneous strict nearly K\"ahler manifold of scalar curvature
        \(\scal\), the minimal polynomial is
        \begin{equation}\label{eq:LJR_NK}
          \lambda^5 \;+\; \tfrac{\scal}{24}\,\|\,\cdot\,\|^2\,\lambda^3
          \;+\; \tfrac{\scal^2}{3600}\,\|\,\cdot\,\|^4\,\lambda
        \end{equation}
\end{enumerate}
\end{example}

However, further examples seem unlikely.
The calculations for generalized Heisenberg groups~\cite{J2} support this conjecture.

\begin{remark} Berger spheres and the compact simply connected homogeneous nearly parallel \(\rmG_2\)-manifolds
  are examples of \(7\)-dimensional normal homogeneous spaces with positive sectional curvature.
  By Example~\ref{ex:LJR}, the choices \(r = \tfrac{4}{\kappa}\) and \(\scal = 378\) yield
  \begin{equation}\label{eq:kappa_=_4_or_scal_=_378} \symR^3(X) + 4\|X\|^2\,\symR^1(X) = 0 \end{equation}
  On the other hand, for each integer \(p \ge 1\) Gallot introduces a partial differential equation
  \((E_p)\) for a function \(f \colon M \to \bbR\) and proves that, if \((M,g)\) is complete and \((E_p)\)
  admits a nonconstant solution, then the universal cover of \(M\) is isometric to the
  round sphere \(\bigl(S^n,\mathrm{can}\bigr)\); see \cite[Proposition~4.1]{Ga}.
  For \(p \coloneqq 2\), Gallot's recursion yields \(\alpha_3(1) = 4\), and \((E_2)\)
  can be written as \[ S^{3}f(X) \coloneqq \nabla^3_{X,X,X}f + \alpha_3(1)\|X\|^2\nabla_X f = 0 \]
  see \cite[p.~248]{Ga} and \cite[Lemme~4.4]{Ga}. Similarly, on a \(6\)-dimensional compact
  homogeneous strict nearly K\"ahler manifold with scalar curvature normalized to \(\scal = 480\),
  we obtain \begin{equation}\label{eq:scal_=_480} \symR^5(X) + 20\|X\|^2\,\symR^3(X) + 64\|X\|^4\,\symR^1(X) = 0
  \end{equation} Even more striking, for \(p \coloneqq 4\) the integers \(\alpha_5(1) = 20\)
  and \(\alpha_5(2) = 64\) appear in \[ S^{5}f(X) \coloneqq \nabla^5_{X,\ldots,X}f
  + \alpha_5(1)\|X\|^2\,\nabla^3_{X,X,X}f + \alpha_5(2)\|X\|^4\,\nabla_X f = 0 \]
  It is known that the standard Riemannian cone
  \[
   (\bar M,\bar g)
   =
   (M \times \bbR_+,\,\d\rho^2 + \rho^2 g)
  \]
  over a \(6\)-dimensional nearly K\"ahler manifold has 
  holonomy contained in \(\rmG_2\) precisely in the usual
  Killing-spinor normalization
  \[
   \scal = n(n-1) = 30
   \qquad (n=6)
  \]
  Equivalently, the corresponding real Killing spinors have
  Killing constant \(|\beta| = \frac12\). Likewise, the standard
  cone over a \(7\)-dimensional nearly parallel \(\rmG_2\)-space
  has holonomy contained in \(\Spin(7)\) in the normalization
  \[
   \scal = n(n-1) = 42
   \qquad (n=7)
  \]
  again corresponding to \(|\beta| = \frac12\). In the proper
  case the cone holonomy is \(\rmG_2\), respectively \(\Spin(7)\);
  in special cases it may be smaller, for example flat,
  \(\SU(4)\), or \(\Sp(2)\).

The normalizations appearing in
\eqref{eq:kappa_=_4_or_scal_=_378} and
\eqref{eq:scal_=_480} are different. Indeed, for a real
Killing spinor satisfying
\[
  \nabla^{\Spin}_X \psi = \beta\,X \cdot \psi
\]
one has
\[
  \scal = 4n(n-1)\beta^2
\]
Thus \(\scal = 378 = 9\cdot 42\) in dimension \(7\) gives
\[
  |\beta| = \frac32
\]
whereas \(\scal = 480 = 16\cdot 30\) in dimension \(6\) gives
\[
  |\beta| = 2
\]
Equivalently, under the scaling \(g \mapsto c^2g\), the
Killing constant transforms as \(\beta \mapsto \beta/c\)
(cf.\ \cite[Thm.~8]{BFGK}). Hence the standard
B\"ar-normalized cone is recovered only after replacing \(g\)
by \(9g\) in the nearly parallel \(\rmG_2\) case, respectively
by \(16g\) in the nearly K\"ahler case. 
Consequently, the scalar curvature normalizations forced by
\eqref{eq:kappa_=_4_or_scal_=_378} and
\eqref{eq:scal_=_480} are not the canonical special-holonomy
cone normalizations.
\end{remark}

\subsection{Estimates for the Singer invariant of a \(\frakC_0\)-space}
Let \((M,g)\) be a Riemannian manifold and let \(p \in M\).
Define a descending sequence of Lie subalgebras
\begin{equation}\label{eq:Lie_algebra_k}
  \frakg(j) \coloneqq \{\, A \in \Endm{T_pM} \mid
  A \cdot \rmR|_p = 0,\; A \cdot \nabla \rmR|_p = 0,\; \ldots,\;
  A \cdot \nabla^j \rmR|_p = 0 \,\}
\end{equation}
where \(\Endm{T_pM}\) is the space of skew-symmetric endomorphisms of \(T_pM\),
and \(A \cdot\) denotes the infinitesimal action of \(A\) on tensors, that is,
the derivation obtained by differentiating at the identity of the natural
\(\SO(T_pM)\)-action by pullback on tensors of mixed type. It follows from the definition that
\(\frakg(j+1) \subseteq \frakg(j)\).

\begin{definition}
Following \cite{GiNi,KT,NTr}, the \emph{Singer invariant} at \(p\) is the
smallest integer \(k_{\mathrm{Singer}} \ge 0\) such that
\[
  \frakg(k_{\mathrm{Singer}})
  = \frakg(k_{\mathrm{Singer}}+1)
\]
\end{definition}

\begin{theorem}\label{th:Singer_invariant}
Let \((M,g)\) be a \(\frakC_0\)-space and let \(p \in M\).
Let \(k_{\mathrm{Singer}}\) be the Singer invariant at \(p\) and let \(k\) be
the degree of \(P_{\min}(M,g,p)\).
Then
\[
  k_{\mathrm{Singer}} \le k
\]
\end{theorem}

\subsection{Concluding remarks}

It would be valuable to have more concrete examples illustrating how
the assertions of Theorem~\ref{th:main} manifest themselves in
practice. Apart from G.O.\ spaces and generalized Heisenberg groups,
no other \(\frakC_0\)-spaces are known. Thus far, all such examples
are homogeneous. 

More generally, it is natural to ask how broad the class of
Riemannian manifolds admitting a minimal polynomial
\(P_{\min}(M,g) \in \Gamma(\scrP(M))[\lambda]\) is and in particular
whether there are further homogeneous examples.

On the other hand, the left-invariant metric on \(\SU(2)\) corresponding to
\((\lambda_1,\lambda_2,\lambda_3)=(2,1,\tfrac12)\) is neither bi-invariant
nor of Berger type. Here, one can show that the minimal polynomial in the sense
of Definition~\ref{de:minimal_polynomial} does not exist;
in particular, the kernel of the evaluation map \(\Kern(\Eval_p)\)
is not an ideal of \(\scrP(T_pM)[\lambda]\). Hence, one cannot hope to characterize general
homogeneous spaces by the concepts developed in this article.

\section{An ideal-theoretic approach to polynomial minimality}
\label{se:ideal}
We work in the following abstract framework.
Let \(\scrR\) be a Noetherian UFD, and let \(\scrM\) be a torsion-free
Noetherian \(\scrR\)-module. Every sequence
\[
  \symR = (\symR^0,\symR^1,\ldots) \in \scrM^{\N_0}
\]
determines a unique \(\scrR\)-module homomorphism
\begin{equation}\label{eq:eval_3}
  \Eval_{\symR} \colon \scrR[\lambda] \longrightarrow \scrM
\end{equation}
such that
\begin{equation}\label{eq:eval_4}
  \Eval_{\symR}(\lambda^i) \coloneqq \symR^i
\end{equation}
for all \(i \in \N_0\), in analogy with~\eqref{eq:eval_1}
and~\eqref{eq:eval_2}. Since \(\Eval_{\symR}\) is an \(\scrR\)-module
homomorphism, its kernel \(\Kern(\Eval_{\symR})\) is automatically
closed under multiplication by elements of \(\scrR\). In general, however,
it need not be closed under multiplication by \(\lambda\), and
therefore it is not automatically an ideal of \(\scrR[\lambda]\).

\begin{lemma}\label{le:principal_ideal}
Let \(\scrR\) be a Noetherian UFD, let \(\scrM\) be a torsion-free
Noetherian \(\scrR\)-module, and let
\[
  \symR = (\symR^0,\symR^1,\ldots) \in \scrM^{\N_0}
\]
If \(\Kern(\Eval_{\symR})\) is an ideal of \(\scrR[\lambda]\), then it is a
principal ideal of \(\scrR[\lambda]\).
\end{lemma}

\begin{proof}
Let
\[
  S \coloneqq \scrR \setminus \{0\}
\]
so that
\[
  \scrK \coloneqq S^{-1}\scrR
\]
is the quotient field of \(\scrR\).

Since \(\scrM\) is a Noetherian \(\scrR\)-module, the ascending chain of
submodules
\[
  \scrM_i \coloneqq \langle \symR^0,\ldots,\symR^i\rangle_\scrR
\]
stabilizes. Hence there exists \(m \in \N\) such that
\[
  \symR^m \in \langle \symR^0,\ldots,\symR^{m-1}\rangle_\scrR
\]
Therefore
\[
  \symR^m + c_1 \symR^{m-1} + \cdots + c_m \symR^0 = 0
\]
for suitable coefficients \(c_1,\ldots,c_m \in \scrR\). Thus the monic
polynomial
\[
  P(\lambda)
  \coloneqq
  \lambda^m + c_1 \lambda^{m-1} + \cdots + c_m
\]
satisfies
\[
  \Eval_{\symR}(P) = 0
\]

Localizing the exact sequence
\[
  0 \longrightarrow \Kern(\Eval_{\symR})
    \longrightarrow \scrR[\lambda]
    \xrightarrow{\Eval_{\symR}} \scrM
\]
at \(S\), we obtain an exact sequence
\[
  0 \longrightarrow S^{-1}\Kern(\Eval_{\symR})
    \longrightarrow \scrK[\lambda]
    \xrightarrow{\widetilde{\Eval}_{\symR}} \scrM \otimes_\scrR \scrK
\]
where \(\widetilde{\Eval}_{\symR}\) denotes the scalar extension of
\(\Eval_{\symR}\). Hence
\[
  \Kern(\widetilde{\Eval}_{\symR})
  =
  S^{-1}\Kern(\Eval_{\symR})
\]
In particular, \(\Kern(\widetilde{\Eval}_{\symR})\) is an ideal of
\(K[\lambda]\). Since \(P \in \Kern(\Eval_{\symR})\), viewing \(P\) in
\(K[\lambda]\) yields
\[
  0 \neq P \in \Kern(\widetilde{\Eval}_{\symR})
\]
Therefore \(\Kern(\widetilde{\Eval}_{\symR})\) is a nonzero ideal of the
PID \(\scrK[\lambda]\), and hence there exists a unique monic polynomial
\[
  P_{\min}
  =
  \sum_{i=0}^{k} a_i \lambda^{k-i}
  \in \scrK[\lambda]
\]
such that
\[
  \Kern(\widetilde{\Eval}_{\symR})
  =
  \scrK[\lambda]\,P_{\min}
\]

Since \(P \in \Kern(\widetilde{\Eval}_{\symR})\), the polynomial
\(P_{\min}\) divides \(P\) in \(\scrK[\lambda]\). Let \(\overline{\scrK}\) be
an algebraic extension of \(\scrK\) containing all roots of \(P\). Since
\(P\) is monic with coefficients in \(\scrR\), every root of \(P\) is
integral over \(\scrR\). Since \(P_{\min}\) divides \(P\) in \(\scrK[\lambda]\),
every root of \(P_{\min}\) in \(\overline{\scrK}\) is also a root of \(P\),
and hence is integral over \(\scrR\). Therefore the coefficients \(a_i\) of \(P_{\min}\),
being elementary symmetric polynomials in the roots of \(P_{\min}\),
are integral over \(\scrR\), too, because integral elements form a
subring; see~\cite[Corollary~5.3]{AM}.

On the other hand, \(a_i \in \scrK\). Since \(\scrR\) is a UFD, it is
integrally closed in its quotient field. Indeed, if \(x = b/c \in \scrK\)
with \(b,c \in \scrR\) coprime is integral over \(\scrR\), then a monic
equation
\[
  x^k = \sum_{i=1}^k \tilde a_i x^{k-i}
\]
for \(x\), with \(\tilde a_i \in \scrR\), implies
\[
  b^k
  =
  \sum_{i=1}^k \tilde a_i b^{k-i}c^i
  =
  c \sum_{i=0}^{k-1} \tilde a_{i+1} b^{k-i-1}c^i
\]
Hence \(c \mid b^k\). Since \(b\) and \(c\) are coprime and \(\scrR\) is a
UFD, it follows that \(c\) is a unit. Thus \(x \in \scrR\).

Applying this to \(x = a_i\), we see that
\[
  a_i \in \scrR
  \qquad
  \text{for all } i
\]
Therefore
\[
  P_{\min} \in \scrR[\lambda]
\]

Since \(\scrM\) is torsion-free, the natural map
\[
  \scrM \longrightarrow \scrM \otimes_\scrR \scrK
\]
is injective. Thus, for \(f \in \scrR[\lambda]\), we have
\[
  \Eval_{\symR}(f) = 0
  \quad\Longleftrightarrow\quad
  \widetilde{\Eval}_{\symR}(f) = 0
\]
and therefore
\[
  \Kern(\Eval_{\symR})
  =
  \Kern(\widetilde{\Eval}_{\symR}) \cap \scrR[\lambda]
  =
  \scrK[\lambda]\,P_{\min} \cap \scrR[\lambda]
\]

It remains to prove that
\[
  \scrK[\lambda]\,P_{\min} \cap \scrR[\lambda]
  =
  \scrR[\lambda]\,P_{\min}
\]
The inclusion “\(\supseteq\)” is clear because \(P_{\min} \in \scrR[\lambda]\).
Conversely, let
\[
  Q \in \scrK[\lambda]\,P_{\min} \cap \scrR[\lambda]
\]
Since \(P_{\min}\) is monic, the division algorithm in \(\scrR[\lambda]\)
yields polynomials \(U,V \in \scrR[\lambda]\) such that
\[
  Q = U\,P_{\min} + V
\]
1with either \(V = 0\) or \(\deg V < \deg P_{\min}\). But
\(Q \in \scrK[\lambda]\,P_{\min}\), so the remainder upon division by
\(P_{\min}\) must be zero in \(\scrK[\lambda]\). Hence \(V = 0\), and thus
\[
  Q = U\,P_{\min} \in \scrR[\lambda]\,P_{\min}
\]
Therefore
\[
  \scrK[\lambda]\,P_{\min} \cap \scrR[\lambda]
  =
  \scrR[\lambda]\,P_{\min}
\]
and consequently
\[
  \Kern(\Eval_{\symR})
  =
  \scrR[\lambda]\,P_{\min}
\]
This proves that \(\Kern(\Eval_{\symR})\) is a principal ideal of
\(\scrR[\lambda]\).
\end{proof}

\subsection{Application to \(\frakC_0\)-spaces}
\label{se:C_0_space}

Let \(V\) and \(W\) be arbitrary finite-dimensional real vector spaces.
We write \(\scrP(V)\) for the \(\bbR\)-algebra of real-valued
polynomial functions on \(V\) and \(\scrP(V;W)\) for the
\(\scrP(V)\)-module of \(W\)-valued polynomial functions on \(V\). In
the applications relevant here, the pertinent case is
\(V \coloneqq T_pM\) and \(W \coloneqq \End{T_pM}\).

Set \(\scrR \coloneqq \scrP(V)\), \(\scrM \coloneqq \scrP(V;W)\), and let
\[
  \symR = (\symR^0,\symR^1,\ldots) \in \scrM^{\N_0}
\]
be a sequence in \(\scrM\). The following lemma gives a simple sufficient condition
for the kernel of the evaluation map
\(\Eval_{\symR}\colon \scrR[\lambda] \to \scrM\), defined by
\eqref{eq:eval_3} and~\eqref{eq:eval_4}, to be an ideal of
\(\scrR[\lambda]\).

\begin{lemma}\label{le:C_0_ideal}
Suppose that for every \(X \in V\) there exists \(C_X \in \End{W}\)
such that, for all \(i \ge 0\),
\begin{equation}\label{eq:C_X_1}
  \symR^{i+1}(X) = C_X\bigl(\symR^i(X)\bigr)
\end{equation}
Then the kernel \(\Kern(\Eval_{\symR})\) of the evaluation map
\(\Eval_{\symR}\colon \scrR[\lambda]\to\scrM\) is an ideal of
\(\scrR[\lambda]\).
\end{lemma}

\begin{proof}
Since \(\Eval_{\symR}\) is a homomorphism of \(\scrR\)-modules,
\(\Kern(\Eval_{\symR})\) is an \(\scrR\)-submodule of \(\scrR[\lambda]\).
Therefore, to prove that \(\Kern(\Eval_{\symR})\) is an ideal, it
suffices to show that it is closed under multiplication by
\(\lambda\).

Let \(j \ge 0\) and let \(b_i \in \scrR\) for \(0 \le i \le j\). Assume
that
\[
  \Eval_{\symR}\Bigl(\sum_{i=0}^{j} b_i\,\lambda^i\Bigr) = 0
\]
Then for all \(X \in V\),
\[
  0 = \Eval_{\symR}\Bigl(\sum_{i=0}^{j} b_i\,\lambda^i\Bigr)(X)
  =
  \sum_{i=0}^{j} b_i(X)\,\symR^i(X)
\]
Applying \(C_X\) and using~\eqref{eq:C_X_1}, we obtain
\begin{align*}
  0
  &=
  C_X\Bigl(\sum_{i=0}^{j} b_i(X)\,\symR^i(X)\Bigr)
  =
  \sum_{i=0}^{j} b_i(X)\,C_X\bigl(\symR^i(X)\bigr) \\
  &\stackrel{\eqref{eq:C_X_1}}{=}
  \sum_{i=0}^{j} b_i(X)\,\symR^{i+1}(X)
  =
  \Eval_{\symR}\Bigl(\sum_{i=0}^{j} b_i\,\lambda^{i+1}\Bigr)(X)
\end{align*}
for all \(X \in V\). Hence
\[
  \Eval_{\symR}\Bigl(\lambda\sum_{i=0}^{j} b_i\,\lambda^i\Bigr)
  =
  \Eval_{\symR}\Bigl(\sum_{i=0}^{j} b_i\,\lambda^{i+1}\Bigr)
  =
  0
\]
which proves closure under multiplication by \(\lambda\). Thus
\(\Kern(\Eval_{\symR})\) is an ideal of \(\scrR[\lambda]\).
\end{proof}

Let \((M,g)\) be a \(\frakC_0\)-space, let \(p \in M\), and let
\(X \in T_pM\). Consider the geodesic \(\gamma\colon I \to M\) with
\(\gamma(0) = p\) and \(\dot\gamma(0) = X\). Let \(\C_\gamma\) be a
\(\frakC_0\)-structure along \(\gamma\) and set
\begin{equation}\label{eq:def_C_X}
  C_X \coloneqq \ad\bigl(\C_\gamma(0)\bigr)
  \colon \End{T_pM} \longrightarrow \End{T_pM}
\end{equation}
where
\begin{equation}\label{eq:ad}
  \ad \colon \End{T_pM} \longrightarrow \End{\End{T_pM}},\qquad
  A \longmapsto \ad(A) \coloneqq [\,A,\,\cdot\,]
\end{equation}
denotes the adjoint action.

\begin{proposition}\label{p:C_0}
Let \((M,g)\) be a \(\frakC_0\)-space, let \(p \in M\), and let
\(X \in T_pM\). The endomorphism
\(C_X \in \End{\End{T_pM}}\) defined by~\eqref{eq:def_C_X} satisfies
\begin{equation}\label{eq:C_X_2}
  \symR^{i+1}\big|_p(X) = C_X\bigl(\symR^i\big|_p(X)\bigr)
\end{equation}
for all \(i \ge 0\).
\end{proposition}

\begin{proof}
Let \(\gamma\) be the geodesic with \(\gamma(0)=p\) and
\(\dot\gamma(0)=X\). Recall that the Jacobi operator along \(\gamma\)
is the endomorphism field
\(\symR_\gamma(t)\in\End{T_{\gamma(t)}M}\) given by
\[
  \symR_\gamma(t) = \rmR(\,\cdot\,,\dot\gamma(t))\dot\gamma(t)
\]
see~\eqref{eq:Jacobi_operator}. We write
\[
  \symR^{(i)}_\gamma \coloneqq
  \Bigl(\tfrac{\nabla}{\d t}\Bigr)^i \symR_\gamma
  \qquad (i \ge 0)
\]
for its iterated covariant derivatives along \(\gamma\). Here
\(\tfrac{\nabla}{\d t}\) acts on endomorphism fields as in
\eqref{eq:Jacobi_operator_abgeleitet} and is iterated recursively by
\[
  \symR^{(0)}_\gamma(t) \coloneqq \symR_\gamma(t),
  \qquad
  \symR^{(j)}_\gamma(t) \coloneqq
  \tfrac{\nabla}{\d t}\symR^{(j-1)}_\gamma(t)
  \quad (j \ge 1)
\]

Since \(\dot\gamma\) is parallel along \(\gamma\), we have
\begin{equation}\label{eq:Jacobi_operator_versus_symmetrized_curvature_tensor}
  \symR^{(i)}_\gamma(t)
  =
  \symR^i\big|_{\gamma(t)}\bigl(\dot\gamma(t)\bigr)
\end{equation}
for all \(i \ge 0\), where the right-hand side is defined by
\eqref{eq:kth_order_symmetrized_curvature_tensor_1} and
\eqref{eq:kth_order_symmetrized_curvature_tensor_2}.

By definition,
\[
  \symR^{(1)}_\gamma(t)
  =
  [\C_\gamma(t),\,\symR^{(0)}_\gamma(t)]
  =
  \ad\bigl(\C_\gamma(t)\bigr)\bigl(\symR^{(0)}_\gamma(t)\bigr)
\]
for all \(t\).

Since \(\tfrac{\nabla}{\d t}\C_\gamma(t) \equiv 0\), the induced
endomorphism field \(\ad(\C_\gamma(t))\) of \(\End{T_{\gamma(t)}M}\) is
parallel as well. Hence \(\tfrac{\nabla}{\d t}\) commutes with
\(\ad(\C_\gamma(t))\). Equivalently, for every endomorphism field
\(A(t)\) along \(\gamma\),
\begin{equation}\label{eq:ad_commutes_with_nabla_t}
  \tfrac{\nabla}{\d t}\bigl(\ad(\C_\gamma(t))A(t)\bigr)
  =
  \ad(\C_\gamma(t))\Bigl(\tfrac{\nabla}{\d t}A(t)\Bigr)
\end{equation}

Using~\eqref{eq:ad_commutes_with_nabla_t} and differentiating the
identity
\[
  \symR^{(1)}_\gamma = \ad(\C_\gamma)\symR^{(0)}_\gamma
\]
an induction on \(i\) yields
\[
  \symR_{\gamma}^{(i+1)}(t)
  =
  \ad\bigl(\C_\gamma(t)\bigr)\bigl(\symR_{\gamma}^{(i)}(t)\bigr)
\]
for all \(i \ge 0\). Evaluating at \(t=0\) and using
\eqref{eq:def_C_X} and
\eqref{eq:Jacobi_operator_versus_symmetrized_curvature_tensor}, we obtain
\[
  \symR^{i+1}\big|_p(X)
  \stackrel{\eqref{eq:Jacobi_operator_versus_symmetrized_curvature_tensor}}{=}
  \symR_\gamma^{(i+1)}(0)
  =
  \ad\bigl(\C_\gamma(0)\bigr)\bigl(\symR_{\gamma}^{(i)}(0)\bigr)
  \stackrel{\eqref{eq:def_C_X},\,\eqref{eq:Jacobi_operator_versus_symmetrized_curvature_tensor}}{=}
  C_X\bigl(\symR^i\big|_p(X)\bigr)
\]
which is~\eqref{eq:C_X_2}.
\end{proof}

\begin{corollary}\label{co:C_0_ideal}
Let \((M,g)\) be a \(\frakC_0\)-space and let \(p \in M\). The kernel
\(\Kern(\Eval_p)\) of the evaluation map~\eqref{eq:eval_2} is a
principal ideal of \(\scrP(T_pM)[\lambda]\).
\end{corollary}

\begin{proof}
By Lemma~\ref{le:C_0_ideal} and Proposition~\ref{p:C_0}, the kernel
\(\Kern(\Eval_p)\) of the evaluation map~\eqref{eq:eval_2} is an ideal
of \(\scrP(T_pM)[\lambda]\).

Set \(V \coloneqq T_pM\) and \(W \coloneqq \End{V}\), and choose a basis
\(E_1,\ldots,E_n\) of \(V\). Then the graded \(\bbR\)-algebra
\(\scrP(V)\) identifies with the polynomial ring
\(\bbR[X_1,\ldots,X_n]\). By Hilbert's Basis Theorem and Gauss's lemma
(cf.~[Ch.~1, Ex.~2(iv)]\cite{AM}), a polynomial ring over a Noetherian
UFD is again a Noetherian UFD. Iterating this argument, we conclude
that \(\bbR[X_1,\ldots,X_n]\), and hence \(\scrP(V)\), is a Noetherian
UFD.

Moreover, \(\scrP(V;W)\) is a finite free \(\scrP(V)\)-module, so
Lemma~\ref{le:principal_ideal} implies that the ideal
\(\Kern(\Eval_p)\) is principal.
\end{proof}

\section{Proof of Theorem~\ref{th:main}}
\label{se:proof_of_theorem_main}
Continuing with the algebraic setting from the previous section, let
\(V\) and \(W\) be finite-dimensional real vector spaces. Every sequence
\[
  \symR = (\symR^0,\symR^1,\ldots) \in \scrP(V;W)^{\N_0}
\]
induces an ascending chain
\[
  \scrA_0(\symR) \;\subseteq\; \scrA_1(\symR) \;\subseteq\;
  \scrA_2(\symR) \;\subseteq\; \cdots
\]
of algebraic subsets of \(V\), defined by
\begin{equation}\label{eq:A_j}
  \scrA_j(\symR) \;\coloneqq\; \{\, X \in V \mid
  \{\symR^0(X),\,\symR^1(X),\,\ldots,\,\symR^j(X)\}
  \text{ is linearly dependent in } W \,\}
\end{equation}
Moreover, each \(\scrA_j(\symR)\) is closed in the Zariski topology.
Since \(W\) is finite-dimensional, we have \(\scrA_j(\symR) = V\) for
all sufficiently large \(j\). Let \(k \ge 0\) be the smallest integer
such that \(\scrA_k(\symR) = V\).

Next, let \(\scrF(V)\) denote the field of rational functions on \(V\),
and let \(\scrF(V;W)\) denote the \(\scrF(V)\)-module of \(W\)-valued
rational functions on \(V\).

\begin{lemma}\label{le:rational}
There exist uniquely determined \(a_i \in \scrF(V)\) for
\(i = 1,\,\ldots,\,k\) such that
\begin{equation}\label{eq:algebraic_minimal_polynomial_1}
  \symR^k \;=\; -\sum_{i=1}^k a_i\,\symR^{k-i}
\end{equation}
holds in \(\scrF(V;W)\).
\end{lemma}

\begin{proof}
Set \(\scrA_j := \scrA_j(\symR)\).
The integer \(k\) above is the smallest index with \(\scrA_k = V\), that is,
the largest index for which \(\scrA_{k-1} \subsetneq V\).
Thus, \(U := V \setminus \scrA_{k-1}\) is nonempty and open in the Zariski
topology and therefore also open and dense in the standard topology.
By definition it follows directly that
\[
  U \;=\; \bigl\{\, X \in V \,\big|\, \{\symR^j(X)\}_{j=0}^{k-1}
  \text{ is linearly independent} \,\bigr\}
\]
On the other hand, the set
\[
  \{\symR^0(X),\,\symR^1(X),\,\ldots,\,\symR^k(X)\}
\]
is linearly dependent for all \(X \in V\) because \(V = \scrA_k\).
Thus, \(k\) is the first {\em common} index of linear dependence in \(W\).

It follows that there exist uniquely determined functions
\(a_i \colon U \to \bbR\) such that
\[
  \symR^k(X) \;=\; -\sum_{i=1}^k a_i(X)\,\symR^{k-i}(X)
\]
holds for all \(X \in U\).
It remains to show that each \(a_i\) is a rational function. This follows
from an argument entirely analog to Cramer's rule in linear algebra: the unknowns \(a_i(X)\)
appear as coefficients in a linear dependence relation between the
\(k+1\) endomorphisms
\(\symR^0(X),\symR^1(X),\ldots,\symR^k(X)\).
Replacing in turn the \((k-i)\)th entry by \(\symR^k(X)\) and passing to
the exterior product \(\Lambda^k W\) yields
\[
  \symR^0(X) \wedge \cdots \wedge \overbrace{\symR^k(X)}^{k-i}
  \wedge \cdots \wedge \symR^{k-1}(X)
  \;=\; -\,a_i(X)\,\symR^0(X) \wedge \cdots \wedge \symR^{k-1}(X)
\]
for each \(X \in U\).
This is the exact analog of replacing one column of a matrix by the
right-hand side vector when applying Cramer's rule to solve a linear
system.
Furthermore, fix an arbitrary \(X_0 \in U\).
By the defining property of \(U\), we have
\[
  \symR^0(X_0) \wedge \cdots \wedge \symR^{k-1}(X_0) \;\neq\; 0
\]
Hence, there exists a linear functional \(\alpha \in \Lambda^k W^*\) such
that
\[
  \alpha\!\big(\symR^0(X_0) \wedge \cdots \wedge \symR^{k-1}(X_0)\big)
  \;\neq\; 0
\]
Evaluating both sides with \(\alpha\) gives
\[
  a_i(X) \;=\; -\,\frac{\alpha\!\big(\symR^0(X) \wedge \cdots \wedge
  \overbrace{\symR^k(X)}^{k-i} \wedge \cdots \wedge
  \symR^{k-1}(X)\big)}
  {\alpha\!\big(\symR^0(X) \wedge \cdots \wedge
  \symR^{k-1}(X)\big)}
\]
with nonzero denominator in a neighborhood of \(X_0\).
This shows that \(a_i\) is a rational function.
\end{proof}

By the preceding discussion, the polynomial
\begin{equation}\label{eq:algebraic_minimal_polynomial_2}
  P_{\min}(\symR)
  \;\coloneqq\;
  \lambda^k \;+\; \sum_{i=1}^k a_i\,\lambda^{k-i}
  \;\in\; \scrF(V)[\lambda]
\end{equation}
associated with any sequence
\[
  \symR \coloneqq (\symR^0,\symR^1,\ldots)
  \in \scrP(V;W)^{\N_0}
\]
is well-defined. In particular, this applies to the symmetrized jet
\[
  \symR|_p \coloneqq (\symR^0|_p,\symR^1|_p,\ldots)
\]
of the curvature tensor at an arbitrary point \(p\) of a Riemannian
manifold \((M,g)\), yielding a polynomial
\begin{equation}\label{eq:algebraic_minimal_polynomial_3}
  \widetilde P_{\min}(M,g,p) \in \scrF(T_pM)[\lambda]
\end{equation}

Since \(\scrP(T_pM;\End{T_pM})\) is a graded
\(\scrP(T_pM)\)-module and
\(\symR^{k-i}|_p \in \scrP^{k+2-i}(T_pM;\End{T_pM})\), comparison of
homogeneous degrees in~\eqref{eq:algebraic_minimal_polynomial_1} shows
that each coefficient \(a_i\) is homogeneous of degree \(i\). Hence
\(\widetilde P_{\min}(M,g,p)\) is monic and homogeneous and therefore satisfies
all requirements of Definition~\ref{de:minimal_polynomial}, except that
its coefficients need not a priori be polynomial. To complete the proof
of Theorem~\ref{th:main}, it remains to show that in the
\(\frakC_0\)-case these coefficients are in fact polynomial.

\begin{proof}[Proof of Theorem~\ref{th:main}]
Let \((M,g)\) be a \(\frakC_0\)-space, let \(p \in M\), and set
\[
  V \coloneqq T_pM,
  \qquad
  W \coloneqq \End{T_pM}
\]
Consider the sequence
\[
  \symR|_p \coloneqq (\symR^0|_p,\symR^1|_p,\ldots)
  \in \scrP(V;W)^{\N_0}
\]

On the one hand, the kernel \(\Kern(\Eval_p)\) of the evaluation map
\eqref{eq:eval_2} is a principal ideal of \(\scrP(V)[\lambda]\) by
Corollary~\ref{co:C_0_ideal}, hence generated by a unique monic polynomial
\[
  P_{\min}(M,g,p)
  =
  \sum_{i=0}^{k} a_i\,\lambda^{k-i}
\]
with coefficients \(a_i \in \scrP(V)\) for \(i=0,\ldots,k\). In
particular, \(a_0 = 1\).

The fact that \(\Eval_p(P_{\min}(M,g,p)) = 0\) means that
\[
  \sum_{i=0}^{k} a_i\,\symR^{k-i}|_p = 0
\]
Hence, for every \(X \in V\),
\[
  \sum_{i=0}^{k} a_i(X)\,\symR^{k-i}|_p(X) = 0
\]
and therefore
\begin{equation}\label{eq:linearly_dependent_set}
  \{\symR^0|_p(X),\symR^1|_p(X),\ldots,\symR^{k}|_p(X)\}
  \subseteq W
\end{equation}
is linearly dependent. Thus \(\scrA_k(\symR|_p) = V\) by~\eqref{eq:A_j}.
On the other hand, by construction, the degree of the polynomial
\[
  \widetilde P_{\min}(M,g,p)  \in \scrF(V)[\lambda]
\]
from~\eqref{eq:algebraic_minimal_polynomial_3} is the smallest integer \(\tilde k\)
that satisfies  \(\scrA_{\tilde k}(\symR|_p) = V\). It follows that \(\tilde k \leq k\).

To finally prove that \(P_{\min}(M,g,p)  =  \widetilde P_{\min}(M,g,p)\), let
\[
  \widetilde{\Eval}_p \colon \scrF(V)[\lambda] \longrightarrow \scrF(V;W)
\]
be the scalar extension of \(\Eval_p\). Since localization is exact, we
have
\[
  \Kern(\widetilde{\Eval}_p)
  =
  \scrF(V)\otimes_{\scrP(V)} \Kern(\Eval_p)
  =
  \scrF(V)[\lambda]\,P_{\min}
\]
As \(\widetilde P_{\min}(M,g,p) \in \Kern(\widetilde{\Eval}_p)\) by
\eqref{eq:algebraic_minimal_polynomial_1}, it follows that \( P_{\min}(M,g,p)\)
divides \( \widetilde P_{\min}(M,g,p) \) in \(\scrF(V)[\lambda]\).  Therefore
\[
  P_{\min}(M,g,p)  =  \widetilde P_{\min}(M,g,p)
\]
because both
\( \widetilde P_{\min}(M,g,p)\) and \(P_{\min}(M,g,p) \) are monic and \(\tilde k \le k\).

Altogether, this shows that \(P_{\min}(M,g,p)\) satisfies Definition~\ref{de:minimal_polynomial}.
\end{proof}

\section{Pointwise minimal polynomials of a \Czero-space}
\label{se:pointwise_properties}
Following our general algebraic approach, we consider the ascending chain
of algebraic subsets \(\scrA_j \subseteq V\) from~\eqref{eq:A_j} associated to a sequence
\(\symR = (\symR^0,\symR^1,\ldots)\) in \(\scrP(V;W)\). Recall that the degree \(k\)
of the polynomial \(P_{\min}(\symR)\) from~\eqref{eq:algebraic_minimal_polynomial_3}
is the smallest integer \(k\) such that \(\scrA_k = V\) holds.

For each \(X \in V\), if \(\symR^j(X) = 0\) for all \(j \ge 0\), set \(k(X) \coloneqq 0\).
Otherwise, define \(1\le k(X) \le k\) to be the unique integer such that
\begin{equation}\label{eq:k(X)}
  X \in \scrA_{k(X)}(\symR) \setminus \scrA_{k(X)-1}(\symR)
\end{equation}
Then the sequence \((\symR^0(X),\symR^{1}(X),\dotsc, )\) has the following property: The set
\[
  \{\,\symR^0(X),\dotsc,\symR^{k(X)-1}(X)\,\}
\]
is linearly independent in \(W\), while the extended set
\[
  \{\,\symR^0(X),\dotsc,\symR^{k(X)}(X)\,\}
\]
is linearly dependent. Thus there exist unique coefficients
\(a_i(X) \in \bbR\) for \(i = 1,\dotsc,k(X)\) such that
\begin{equation}\label{eq:pointwise_minimal_polynomial_1}
  \symR^{k(X)}(X)
  \;=\; -\sum_{i=1}^{k(X)} a_i(X)\,\symR^{k(X)-i}(X)
\end{equation}
The (pointwise) minimal polynomial of the \(W\)-valued sequence
\[
  \symR(X) \coloneqq (\symR^0(X),\symR^1(X),\ldots)
\]
is defined by
\begin{equation}\label{eq:pointwise_minimal_polynomial_2}
  P_{\min}(\symR(X)) \;\coloneqq\; \lambda^{k(X)} +
  \sum_{i=1}^{k(X)} a_i(X)\,\lambda^{k(X)-i}
\end{equation}

\begin{lemma}\label{le:C_0_roots}
Let \(X \in V\) and suppose there exists \(C_X \in \End{W}\) such that
\eqref{eq:C_X_1} holds for all \(i \ge 0\). Then:
\begin{enumerate}
  \item The pointwise minimal polynomial \(P_{\min}(\symR(X))\) divides the
    minimal polynomial \(P_{\min}(C_X)\) of \(C_X\) in \(\bbR[\lambda]\).
  \item If \(W\) is a Euclidean space and \(C_X \in \Endm{W}\), then all
    complex roots of \(P_{\min}(\symR(X))\) from
    \eqref{eq:pointwise_minimal_polynomial_2} are purely imaginary and
    simple.
\end{enumerate}
\end{lemma}

\begin{proof}
View \(W\) as an \(\bbR[\lambda]\)-module via
\(\lambda \cdot Y \coloneqq C_X(Y)\) for all \(Y \in W\). For
\(\symR^0(X) \in W\), the annihilator
\[
  \Ann_{\bbR[\lambda]}(\symR^0(X)) \;\coloneqq\;
  \{\,P \in \bbR[\lambda] \mid P \cdot \symR^0(X) = 0\,\}
\]
is an ideal of \(\bbR[\lambda]\). Since \(\bbR[\lambda]\) is a principal
ideal domain, \(\Ann_{\bbR[\lambda]}(\symR^0(X))\) is principal and
generated by a unique monic polynomial, which we denote by
\(P_{\min}(C_X,\symR^0(X)) \in \bbR[\lambda]\). Equivalently,
\(P_{\min}(C_X,\symR^0(X))\) is the monic polynomial of smallest degree
such that
\[
  P_{\min}(C_X,\symR^0(X)) \cdot \symR^0(X) = 0
\]
and is called the \emph{minimal annihilating polynomial} of \(\symR^0(X)\)
for \(C_X\).

Consider the cyclic subspace
\[
  \bbR[\lambda]\symR^0(X)
  \;=\; \operatorname{span}\{\,\symR^0(X),\,C_X\symR^0(X),\,
  (C_X)^2\symR^0(X),\ldots\} \subseteq W.
\]
This is the smallest \(C_X\)-invariant subspace of \(W\) containing
\(\symR^0(X)\). Its dimension is the first index at which linear
dependence occurs, i.e. the smallest integer \(\tilde{k}(X)\) such that
\[
  (C_X)^{\tilde{k}(X)}\symR^0(X) \in
  \operatorname{span}\{\,\symR^0(X),\,C_X\symR^0(X),\,\dotsc,\,
  (C_X)^{\tilde{k}(X)-1}\symR^0(X)\,\}
  \footnote{The space \(\bbR[\lambda]\symR^0(X)\)
  is the stabilized Krylov subspace associated with the pair
  \((C_X,\symR^0(X))\), and its dimension \(\tilde{k}(X)\) is
  also called the grade of \(\symR^0(X)\) with respect to \(C_X\).}
\]
Hence there exist uniquely determined \(\tilde a_i(X)\in\bbR\) for
\(i=1,\dotsc,\tilde{k}(X)\) with
\begin{equation}\label{eq:pointwise_minimal_polynomial_3}
  (C_X)^{\tilde{k}(X)}\symR^0(X)
  \;=\; -\sum_{i=1}^{\tilde{k}(X)} \tilde a_i(X)\,
  (C_X)^{\tilde{k}(X)-i}\symR^0(X)
\end{equation}
and thus
\[
  P_{\min}\bigl(C_X,\symR^0(X)\bigr)
  \;=\; \lambda^{\tilde{k}(X)} +
  \sum_{i=1}^{\tilde{k}(X)} \tilde a_i(X)\,\lambda^{\tilde{k}(X)-i}
\]
Using~\eqref{eq:C_X_1}, we have
\[
  (C_X)^i\symR^0(X) \;=\; \symR^i(X)
  \quad\text{for all } i \ge 0
\]
Substituting this into~\eqref{eq:pointwise_minimal_polynomial_3} and
comparing with~\eqref{eq:pointwise_minimal_polynomial_1} yields
\(\tilde a_i(X) = a_i(X)\), \(\tilde{k}(X)=k(X)\), and hence
\[
  P_{\min}\bigl(C_X,\symR^0(X)\bigr) = P_{\min}(\symR(X))
\]

Now let \(P_{\min}(C_X)\) denote the minimal polynomial of the operator
\(C_X\) on \(W\). Then \(P_{\min}(C_X)\) annihilates \(W\), hence also
annihilates \(\symR^0(X)\), i.e.,
\(P_{\min}(C_X)\cdot \symR^0(X)=0\). Therefore
\(P_{\min}(C_X)\in \Ann_{\bbR[\lambda]}(\symR^0(X))\). Since
\(\Ann_{\bbR[\lambda]}(\symR^0(X)) =
 \bbR[\lambda]\,P_{\min}(C_X,\symR^0(X))\), it follows that
\(P_{\min}(C_X,\symR^0(X))\) divides \(P_{\min}(C_X)\). Using the identity
\(P_{\min}(C_X,\symR^0(X)) = P_{\min}(\symR(X))\), we obtain that
\(P_{\min}(\symR(X))\) divides \(P_{\min}(C_X)\). This proves~(a).

For~(b), assume that \(W\) is Euclidean and \(C_X \in \Endm{W}\). Then
\(C_X\) is skew-adjoint, hence normal, and therefore diagonalizable over
\(\bbC\) with purely imaginary eigenvalues. In particular,
\(P_{\min}(C_X)\) has only purely imaginary complex roots, and these roots
are simple. Since \(P_{\min}(\symR(X))\) divides \(P_{\min}(C_X)\) by~(a),
its complex roots are a subset of those of \(P_{\min}(C_X)\), and thus are
purely imaginary and simple as well. This proves~(b).
\end{proof}

\begin{corollary}\label{co:C_0_2}
Let \((M,g)\) be a \(\frakC_0\)-space with positive definite metric.
Let \(p \in M\) and \(X \in T_pM\). Then all complex roots of the minimal
polynomial \(P_{\min}(\symR|_p(X))\) from
\eqref{eq:pointwise_minimal_polynomial_2} are purely imaginary and simple.
\end{corollary}

\begin{proof}
Let \(\gamma\) be the geodesic with \(\gamma(0)=p\) and \(\dot\gamma(0)=X\).
Set \(W \coloneqq \Endp{T_pM}\). Since \(\C_\gamma(0)\) is \(g_p\)-skew-adjoint,
the commutator map \(\ad\bigl(\C_\gamma(0)\bigr)\) preserves \(W\): for
\(A \in W\),
\[
  \ad\bigl(\C_\gamma(0)\bigr)(A)^{*}
  \;=\; \bigl[\C_\gamma(0),A\bigr]^{*}
  \;=\; \bigl[\C_\gamma(0),A\bigr]
\]
Moreover, \(\symR|_p(X)\in W\) for all \(X \in T_pM\) by~\eqref{eq:scrR_is_self_adjoint}.
Equip \(\End{T_pM}\) with the inner product
\(\langle A,B\rangle \coloneqq \trace(A\circ B)\), which restricts to a
positive definite inner product on \(W\). Then
\[
  C_X \coloneqq \ad\bigl(\C_\gamma(0)\bigr)|_W \colon W \longrightarrow W
\]
is skew-adjoint with respect to \(\langle\,\cdot\,,\,\cdot\,\rangle\), since
for all \(A,B \in W\),
\[
  \langle C_XA,B\rangle
  \;=\; \trace\bigl(\,[\C_\gamma(0),A]\circ B\,\bigr)
  \;=\; -\trace\bigl(\,A\circ[\C_\gamma(0),B]\,\bigr)
  \;=\; -\langle A,C_XB\rangle
\]
Also~\eqref{eq:C_X_1} holds for this \(C_X\) according to~\eqref{eq:def_C_X},\eqref{eq:C_X_2}.
Hence Lemma~\ref{le:C_0_roots} applies (with \(V=T_pM\) and \(W=\Endp{T_pM}\))
and yields that all complex roots of \(P_{\min}(\symR|_p(X))\) are purely
imaginary and simple.
\end{proof}

\subsection{Proofs of Theorem~\ref{th:main_addition} and
Corollary~\ref{co:main_addition}}
\label{se:proof_of_theorem_main_addition}

\begin{proof}
For each \(X \in T_pM\), let \(k(X)\) be the degree of the minimal
polynomial \(P_{\min}(\symR|_p(X))\) defined in
\eqref{eq:pointwise_minimal_polynomial_2}. The subset \(U(p) \subseteq T_pM\)
from~\eqref{eq:U(p)_1} admits the equivalent characterizations:
\begin{align}
  U(p) &= \{\,X \in T_pM \mid k(X) = k\,\} \\
  \label{eq:U(p)_3}
       &= \{\,X \in T_pM \mid
  P_{\min}(M,g,p)(X) = P_{\min}(\symR|_p(X))\,\}
\end{align}

By \eqref{eq:U(p)_3}, \(P_{\min}(M,g,p)(X) = P_{\min}(\symR|_p(X))\) for
each \(X \in U(p)\). Hence
Corollary~\ref{co:C_0_2} shows that, in the positive definite case, all
complex roots of \(P_{\min}(M,g,p)(X)\) are purely imaginary and simple
for all \(X \in U(p)\). This gives Theorem~\ref{th:main_addition}~(a). 

Now suppose that \(a_k \not = 0\). We claim that \(\Ric|_p = 0\).
Consider again the geodesic \(\gamma \colon I \to M\) with \(\gamma(0)=p\)
and \(\dot\gamma(0)=X \in U(p)\). Since the Ricci tensor of a
\(\frakC_0\)-space is obviously a Killing tensor (cf.\ also~\cite[Theorem~3]{BV}), the
function
\[
  t \longmapsto \Ric(\dot\gamma(t),\dot\gamma(t))
\]
is constant along \(\gamma\). Hence
\[
  \trace\bigl(\symR_\gamma^{(i)}\bigr)(0)
  = \left.\frac{\mathrm{d}^i}{\mathrm{d}t^i}\right|_{t=0}
    \Ric(\dot\gamma(t),\dot\gamma(t))
  = 0
\]
for all \(i \ge 1\). Therefore, taking the trace of
\eqref{eq:admissible_in_p}, the left-hand side and all terms on the
right-hand side vanish except the summand with \(i = k\). Evaluating at
\(t = 0\) yields
\[
  0 = -\,a_k(X)\,\Ric(X,X)
\]
for all \(X \in U(p)\). By density of \(U(p)\subseteq T_pM\), we conclude
that \(a_k\,\Ric|_p = 0\). Because the polynomial ring \(\scrP(T_pM)\)
has no zero divisors, \(a_k \not = 0\) implies that \(\Ric|_p = 0\).
This finishes the proof of Theorem~\ref{th:main_addition}~(b).

For Corollary~\ref{co:main_addition}, assume again that the metric tensor
is positive definite. Let \(X \in U(p)\). Write the nonzero roots of
\(P_{\min}(M,g,p)(X)\) as
\(\{\pm \mathrm{i}\lambda_1(X),\dotsc,\pm \mathrm{i}
\lambda_{\lfloor k/2 \rfloor}(X)\}\), where \(\lambda_i(X) > 0\) and
\(\lambda_i(X) \neq \lambda_j(X)\) for \(i \ne j\). Also the only possible
real root is \(0\). Accordingly, note that \(0\) is a root if and only if
\(k\) is odd. Therefore, we have
\begin{equation*}
\begin{aligned}
  P_{\min}(M,g,p)(X)
  &=
  \begin{cases}
    \lambda\,\prod_{i=1}^{\lfloor k/2 \rfloor}
      \bigl(\lambda^2 + \lambda_i(X)^2\bigr)
      &\text{if \(k\) is odd} \\
    \quad\prod_{i=1}^{\lfloor k/2 \rfloor}
      \bigl(\lambda^2 + \lambda_i(X)^2\bigr)
      &\text{if \(k\) is even}
  \end{cases} \\
  &=
  \sum_{i=0}^{\lfloor k/2 \rfloor} a_{2i}(X)\,\lambda^{\,k-2i}
\end{aligned}
\end{equation*}
with \(a_{2i}(X) = \sigma_i\bigl(\lambda_1(X)^2,\ldots,
\lambda_{\lfloor k/2 \rfloor}(X)^2\bigr)\), where \(\sigma_i\) denotes the \(i\)th
elementary symmetric polynomial. Thus \(a_{2i} > 0\) on \(U(p)\) and hence
\(a_{2i} \geq 0\) on \(T_pM\) because \(U(p)\) is dense in \(T_pM\). This
establishes~\eqref{eq:min_pol_pos_def}. In particular, if \(k\) is even, then
\[
  a_k(X)
  = \lambda_1(X)^2\cdot\cdots \cdot\lambda_{k/2}(X)^2
  > 0
\]
on \(U(p)\). Hence \(\Ric|_p = 0\) by
Theorem~\ref{th:main_addition}~(b), which finishes the proof
of Corollary~\ref{co:main_addition}.
\end{proof}

\section{Proof of Theorem~\ref{th:main_global}}
\label{se:further_properties}
\begin{proof}
  For~(a), let us first show that the \(a_i\) define Killing tensors on \(U(M)\).
  By the explicit construction in the proof of Lemma~\ref{le:rational}, the
  pointwise-defined coefficients \(\bigl.a_i\bigr|_p\) of the minimal
  polynomial \(P_{\min}(M,g,p)\) fit together smoothly to define the polynomial
  \[
    P_{\min}\bigl(U(M),g|_{U(M)}\bigr)
    =
    \sum_{i=0}^k a_i\,\lambda^{k-i}
    \in \Gamma\bigl(\scrP(U(M))\bigr)[\lambda]
  \]
  such that
  \[
    \sum_{i=0}^k a_i\,\symR^{k-i} = 0
    \qquad \text{on } U(M)
  \]

Applying \(\nabla_X\) and evaluating at \(p \in U(M)\) and \(X \in T_pM\)
\begin{align*}
  0
  &=
  \nabla_X \biggl(\sum_{i=0}^k a_i\,\symR^{k-i}\biggr)\bigg|_p(X) \\
  &=
  \sum_{i=0}^k \bigl(\nabla_X(a_i\,\symR^{k-i})\bigr)\big|_p(X) \\
  &=
  \sum_{i=0}^k
  (\nabla_X a_i)\big|_p(X)\,\symR^{k-i}\big|_p(X)
  +
  \sum_{i=0}^k
  a_i\big|_p(X)\,
  \underbrace{(\nabla_X \symR^{k-i})\big|_p(X)}_{=\symR^{k-i+1}\big|_p(X)}
\end{align*}
by the product rule. Hence
\[
  \sum_{i=0}^k
  (\nabla_X a_i)\big|_p(X)\,\symR^{k-i}\big|_p(X)
  =
  -\sum_{i=0}^k
  a_i\big|_p(X)\,\symR^{k-i+1}\big|_p(X)
\]

On the other hand,
\[
  \lambda \sum_{i=0}^k a_i\big|_p(X)\,\symR^{k-i}\big|_p(X) = \sum_{i=0}^k a_i\big|_p(X)\,\symR^{k-i+1}\big|_p(X) = 0
\]
because \(\lambda P_\min(M,g,p) \in \Kern(\Eval_p)\)
by the ideal property, cf. Corollary~\ref{co:C_0_ideal}. Also,
\(a_0 = 1\). Therefore,
\[
  \sum_{i=1}^k
  (\nabla_X a_i)\big|_p(X)\,\symR^{k-i}\big|_p(X) = 0
\]

Since
\(\{\symR^0|_p(X),\symR^1|_p(X),\ldots,\symR^{k-1}|_p(X)\}\)
is linearly independent for \(X \in U(p)\), we conclude that
\[
  (\nabla_X a_i)\big|_p(X) = 0
\]
for \(i=1,\ldots,k\). Because \(U(p)\) is dense in \(T_pM\), it follows
that
\[
  (\nabla_X a_i)\big|_p(X) \equiv 0
\]
Therefore, \(a_i|_{U(M)}\) defines a Killing tensor for each \(i\).
\end{proof}

For the proof of Part~(b), we first need to recall the concept of constant
Jacobi osculating rank from~\cite{AN},~\cite{AAN}.
Let \(I\subseteq\bbR\) be an open interval and let
\(\gamma\colon I\to M\) be a geodesic of a \(\frakC_0\)-space \((M,g)\).

There exists an integer \(k\geq 1\) such that the family
\begin{equation}\label{eq:Jacobi_osculating_rank_1}
   \{\symR^j_\gamma(t)\}_{j=0}^{k-1}
\end{equation}
consisting of the Jacobi operator and its first \(k-1\) covariant
derivatives is linearly independent in \(\End{T_{\gamma(t)}M}\) for all
\(t\in I\), and there exist real numbers \(a_1,\ldots,a_k\) such that
\begin{equation*}
   \symR^k_\gamma(t)
   =
   -\sum_{j=1}^k a_j\symR^{k-j}_\gamma(t)
\end{equation*}
for all \(t\in I\). This fact was mentioned in~\cite{AN},~\cite{AAN}
for G.O. spaces, but it holds equally for the larger class of
\(\frakC_0\)-spaces; see also
Proposition~1 of~\cite[Sec.~2.9]{BTV} and~\cite[Proof of Theorem~2.3]{T}. In general, the integer
\(k=k(\gamma)\) depends on the chosen geodesic \(\gamma\). By
Definition~\ref{de:minimal_polynomial}
and~\eqref{eq:Jacobi_operator_versus_symmetrized_curvature_tensor}, it is
equal to the degree \(k(p)\) of the minimal polynomial at \(p\) if and
only if \(X\in U(p)\), where we put
\[
p\coloneq\gamma(0)
\qquad\text{and}\qquad
X\coloneq\dot\gamma(0).
\]

\bigskip

\begin{lemma}\label{le:radially_invariant}
Let \(p\in U(M)\) and \(X\in U(p)\). Then the geodesic
\[
    \gamma_X(t)=\exp_p(tX)
\]
is entirely contained in \(U(M)\), for all \(t\) for which it is defined.
In particular,
\[
    \exp_p(U(p))\subseteq U(M).
\]
\end{lemma}

\begin{proof}
Since \(X\in U(p)\) and \(p\in U(M)\), we have
\[
    k(\gamma_X)=k(p)=k(M).
\]
For every \(t\) in the domain of \(\gamma_X\), put
\[
    q\coloneq\gamma_X(t).
\]
Regarding \(\gamma_X\) as a geodesic through \(q\) gives
\[
    k(M)=k(\gamma_X)\leq k(q)\leq k(M).
\]
Hence \(k(q)=k(M)\), and therefore \(q\in U(M)\). Thus
\(\gamma_X(t)\in U(M)\) for all \(t\) for which \(\gamma_X(t)\) is
defined. In particular,
\[
    \exp_p(U(p))\subseteq U(M).
\]
\end{proof}

\subsection{Proof of Theorem~\ref{th:main_global}~(b)}

\begin{proof}
Here we assume that \((M,g)\) is real-analytic. Because \(M\setminus U(M)\)
is a proper real-analytic subset, \(U(M)\subseteq M\) is open and dense.
We claim that the coefficients \(a_i\) of the minimal polynomial on
\(U(M)\) extend to Killing tensors on all of \(M\).

Since the Killing equation is of finite type, Killing tensors correspond
bijectively to parallel sections of a naturally defined vector bundle
\((\bbE,\nabla^{\bbE})\) over \(M\), called the \emph{prolongation bundle}
of the Killing equation. In the real-analytic case, this bundle and its
connection are real-analytic; only smoothness of the connection will be
used below.

Let
\[
    \psi\in\Gamma(\bbE|_{U(M)})
\]
be the parallel section corresponding to a Killing tensor
\(a\) defined on \(U(M)\). We show that \(\psi\) extends uniquely to a
globally defined smooth \(\nabla^{\bbE}\)-parallel section of \(\bbE\).

Fix an arbitrary point \(q\in M\). By the standard existence theorem for
totally normal neighborhoods, there is an open neighborhood \(V\) of \(q\)
which is a normal neighborhood of each of its points. Since \(U(M)\) is
dense, we may choose
\[
    p\in V\cap U(M).
\]
By total normality, there is a star-shaped open neighborhood
\(\Omega\subseteq T_pM\) of the origin such that
\[
    \exp_p\colon\Omega\longrightarrow V
\]
is a diffeomorphism.

For \(X\in\Omega\), let
\[
    \gamma_X(t)=\exp_p(tX),
    \qquad
    0\leq t\leq 1,
\]
and define a section \(\widehat\psi_p\) of \(\bbE|_V\) by
\[
    \widehat\psi_p(\exp_pX)
    \coloneq
    P^{\bbE}_{\gamma_X,0\to1}\bigl(\psi(p)\bigr),
\]
where \(P^{\bbE}\) denotes \(\nabla^{\bbE}\)-parallel transport. Since
\(\nabla^{\bbE}\) and the exponential map are smooth,
\(\widehat\psi_p\) depends smoothly on \(X\) and hence defines a smooth
section on \(V\).

If \(X\in U(p)\cap\Omega\), then
\[
    \gamma_X([0,1])\subseteq U(M)
\]
by Lemma~\ref{le:radially_invariant}. Since \(\psi\) is
\(\nabla^{\bbE}\)-parallel on \(U(M)\), it follows that
\[
    \widehat\psi_p(\exp_pX)
    =
    \psi(\exp_pX).
\]
Thus \(\widehat\psi_p\) and \(\psi\) agree on
\[
    A_p
    \coloneq
    \exp_p\bigl(U(p)\cap\Omega\bigr).
\]
Because \(U(p)\) is dense in \(T_pM\) and \(\exp_p\) is a diffeomorphism
on \(\Omega\), the subset \(A_p\) is dense in \(V\). Moreover,
\[
    A_p\subseteq U(M)\cap V
\]
by Lemma~\ref{le:radially_invariant}. Hence \(A_p\) is dense in
\(U(M)\cap V\). Since both \(\widehat\psi_p\) and \(\psi\) are continuous
on \(U(M)\cap V\), it follows that
\[
    \widehat\psi_p|_{U(M)\cap V}
    =
    \psi|_{U(M)\cap V}.
\]
Thus \(\widehat\psi_p\) is a smooth extension of
\(\psi|_{U(M)\cap V}\) to all of \(V\).

Moreover,
\[
    \nabla^{\bbE}\widehat\psi_p=0
\]
on the open dense subset \(U(M)\cap V\), since there
\(\widehat\psi_p=\psi\). Since \(\widehat\psi_p\) is smooth,
\(\nabla^{\bbE}\widehat\psi_p\) is continuous, and therefore
\[
    \nabla^{\bbE}\widehat\psi_p=0
    \qquad\text{on }V.
\]

Since \(q\in M\) was arbitrary, this construction gives a smooth
\(\nabla^{\bbE}\)-parallel extension of \(\psi\) to a neighborhood of
every point of \(M\). These local extensions agree on overlaps. Indeed,
on every nonempty overlap \(V_1\cap V_2\), the set
\[
    U(M)\cap V_1\cap V_2
\]
is dense, and both local extensions agree there with \(\psi\). By
continuity, they therefore agree on all of \(V_1\cap V_2\). Hence the
local extensions glue to a globally defined smooth section
\[
    \bar\psi\in\Gamma(\bbE)
\]
satisfying
\[
    \nabla^{\bbE}\bar\psi=0
    \qquad\text{on }M.
\]

The zeroth component of \(\bar\psi\) is consequently a smooth global
Killing tensor extending the original Killing tensor \(a\). Applying
this argument to each coefficient \(a_i\) proves the claim.
\end{proof}

\section{Proof of Theorem~\ref{th:Singer_invariant}}
\label{se:proof_of_theorem_Sin_inv}

Returning to the algebraic setup from
Section~\ref{se:C_0_space}, let \(V\) and \(W\) be arbitrary
finite-dimensional real vector spaces, and let
\[
  \symR = (\symR^0,\symR^1,\ldots) \in \scrP(V;W)^{\N_0}
\]
be a sequence of \(W\)-valued polynomial functions.

Let
\[
  P_{\min}(\symR) = \lambda^k + \sum_{i=1}^k a_i\,\lambda^{k-i}
  \in \scrF(V)[\lambda]
\]
denote the polynomial from~\eqref{eq:algebraic_minimal_polynomial_2}.
In the geometric application below, Theorem~\ref{th:main} ensures that
all coefficients lie in \(\scrP(V)\). Accordingly, throughout this
section we restrict attention to the case
\[
  P_{\min}(\symR) \in \scrP(V)[\lambda]
\]

Suppose that \(W\) carries a representation
\[
  \rho\colon \GL(V) \to \GL(W)
\]
and let \(\GL(V)\) act on \(\scrP(V;W)\) by
\begin{equation}\label{eq:Lie_group_action}
  (F \cdot \Phi)(X) \coloneqq \rho(F)\bigl(\Phi(F^{-1}X)\bigr)
\end{equation}
for \(\Phi \in \scrP(V;W)\). This induces a diagonal action on
sequences in \(\scrP(V;W)^{\N_0}\). The stabilizer of the truncated
sequence
\[
  \symR^{\le k} \coloneqq (\symR^0,\symR^1,\ldots,\symR^k)
\]
in \(\GL(V)\) is
\begin{equation}\label{eq:j-te_Lie_gruppe_alternativ}
  G(\symR^{\le k})
  \coloneqq
  \{\,F \in \GL(V) \mid
      F \cdot \symR^i = \symR^i \text{ for } i=0,\ldots,k
    \,\}
\end{equation}
Moreover, \(\GL(V)\) acts on \(\scrP(V)\) by pullback:
\[
  (F\cdot a)(X) \coloneqq a(F^{-1}X)
\]
If \(P = \sum_j a_j\lambda^j \in \scrP(V)[\lambda]\), we define
\[
  F \cdot P \coloneqq \sum_j (F\cdot a_j)\lambda^j
\]

\begin{lemma}\label{le:Sin_inv}
Each coefficient \(a_i\) of \(P_{\min}(\symR)\) is invariant under the
induced action:
\[
  F\cdot a_i = a_i
\]
for all \(F \in G(\symR^{\le k})\) and \(i=1,\ldots,k\).
\end{lemma}

\begin{proof}
  Let \(F \in G(\symR^{\le k})\) and set \(P_{\min} \coloneq P_{\min}(\symR)\).
  Then \(F \cdot \symR^i = \symR^i\) for
\(i=0,\ldots,k\). By~\eqref{eq:algebraic_minimal_polynomial_1}, we have
\[
  \Eval_{\symR}(P_{\min}) = 0
\]
Applying \(F\cdot\) and using the equivariance relation
\[
  F\cdot \Eval_{\symR}(P)
  =
  \Eval_{F\cdot\symR}(F\cdot P)
\]
for \(P \in \scrP(V)[\lambda]\), we obtain
\[
  0
  =
  F\cdot \Eval_{\symR}(P_{\min})
  =
  \Eval_{F\cdot\symR}(F\cdot P_{\min})
\]
Since \(F \cdot \symR^i = \symR^i\) for \(i=0,\ldots,k\) and
\(\deg(P_{\min}) = k\), it follows that
\[
  \Eval_{\symR}(F\cdot P_{\min}) = 0
\]
Thus \(F\cdot P_{\min}\) is a monic polynomial of degree \(k\) in
\(\scrP(V)[\lambda]\) annihilating \(\symR\). By uniqueness of
\(P_{\min}(\symR)\), we obtain
\[
  F\cdot P_{\min} = P_{\min}
\]
and therefore \(F\cdot a_i = a_i\) for \(i=1,\ldots,k\).
\end{proof}

\begin{corollary}\label{co:Sin_inv}
Suppose that \(\Kern(\Eval_{\symR})\) is an ideal of
\(\scrP(V)[\lambda]\). Then
\[
  G(\symR^{\le k}) = G(\symR^{\le k+1})
\]
\end{corollary}

\begin{proof}
The inclusion
\[
  G(\symR^{\le k+1}) \subseteq G(\symR^{\le k})
\]
is obvious, so it suffices to show the reverse inclusion.

Set \(P_{\min} \coloneqq P_{\min}(\symR)\). Since
\(\Kern(\Eval_{\symR})\) is an ideal of \(\scrP(V)[\lambda]\), we have
in particular
\[
  \lambda P_{\min} \in \Kern(\Eval_{\symR})
\]
Applying \(\Eval_{\symR}\) to \(\lambda P_{\min}\), we obtain
\[
  \symR^{k+1} = -\sum_{i=1}^k a_i\,\symR^{k+1-i}
\]
Now let \(F \in G(\symR^{\le k})\). By Lemma~\ref{le:Sin_inv}, each
coefficient \(a_i\) is \(F\)-invariant, and by definition of
\(G(\symR^{\le k})\), each \(\symR^{k+1-i}\) is \(F\)-invariant for
\(i=1,\ldots,k\). Hence the right-hand side is \(F\)-invariant, and so
\(\symR^{k+1}\) is \(F\)-invariant as well. Therefore
\[
  G(\symR^{\le k}) \subseteq G(\symR^{\le k+1})
\]
This completes the proof.
\end{proof}

\paragraph{Proof of Theorem~\ref{th:Singer_invariant}}
Fix \(p \in M\) and let \(k\) denote the degree of the minimal
polynomial \(P_{\min}(M,g,p)\). Put \(V \coloneqq T_pM\),
\(W \coloneqq \End{T_pM}\), and let \(\rho\) denote
the conjugation representation
\[
  \rho \colon \GL(V) \longrightarrow \GL(W),
  \quad
  F \longmapsto \Ad(F),
  \qquad
  \Ad(F)(A) \coloneqq F A F^{-1}
\]
Under the identification of \(\symR^i|_p\) with an
\(W\)-valued polynomial function on \(V\), the natural
\(\SO(V)\)-action by pullback corresponds to the action \(F\cdot\)
in~\eqref{eq:Lie_group_action} with \(\rho = \Ad\).

Hence, on the one hand, by the classical jet isomorphism theorem
(cf.~\cite[Theorem~3]{J1}), the stabilizer in \(\SO(V)\) of the
\(k\)th-order jet \(\nabla^{\le k}\rmR|_p\) of the curvature tensor
\begin{equation}\label{eq:Lie_group_k}
  G(k)
  =
  \{\,F \in \SO(V)  \mid
      F^* \nabla^i \rmR|_p = \nabla^i \rmR|_p \text{ for } i=0,\ldots,k
    \,\}
\end{equation}
coincides with the stabilizer of the symmetrized \(k\)-jet
\(\symR^{\le k}|_p\)
\begin{equation}\label{eq:Lie_group_k_alternativ}
  G(k)
  =
  \{\,F \in \SO(V)  \mid
      F \cdot \symR^i|_p = \symR^i|_p \text{ for } i=0,\ldots,k
    \,\}
\end{equation}

On the other hand, Corollary~\ref{co:C_0_ideal} implies
that \(\Kern(\Eval_p)\) is an ideal of \(\scrP(V)[\lambda]\).
Hence, by Corollary~\ref{co:Sin_inv}, we have
\[
  G_{\GL}(k) = G_{\GL}(k+1)
\]
Since \(G(k) = G_{\GL}(k) \cap \SO(V) \) by
\eqref{eq:Lie_group_k_alternativ}, we conclude that
\[
  G(k) = G(k+1)
\]
In particular, \(\frakg(k) = \frakg(k+1)\). Hence
\[
  k_{\mathrm{Singer}} \le k
\]
\qed

\appendix
\section{Existence of admissible polynomials on compact real-analytic
Riemannian manifolds}\label{se:JR}


\begin{lemma} \label{le:admissible}
Let \((M,g)\) be a Riemannian manifold with Levi-Civita connection
\(\nabla\) and curvature tensor \(\rmR\). For each \(p \in M\) there exist
some \(k > 0\) and symmetric tensors \(a_i \in \scrP^i(T_pM)\) for
\(i = 1,\ldots,k\) such that \eqref{eq:admissible_in_p} holds at \(p\).
\end{lemma}

\begin{proof}
  Recall that \(\scrP(T_pM;\End{T_pM})\) is a free \(\scrP(T_pM)\)-module of
  rank \(n\times n\). In particular,
\(\scrP(T_pM;\End{T_pM})\) satisfies the ascending chain condition.
It follows that there exist \(k > 0\) and \(a_i \in \scrP(T_pM)\) such that
\[
  \symR^k|_p = -a_1\,\symR^{k-1}|_p - a_2\,\symR^{k-2}|_p
      - \cdots - a_k\,\symR^0|_p
\]
cf. the proof of Lemma~\ref{le:principal_ideal}. Since \(\scrP(T_pM;\End{T_pM})\) is
a graded \(\scrP(T_pM)\)-module and \(\symR^k\in \scrP^{k+2}(T_pM;\End{T_pM})\),
we can furthermore assume that \(a_i \in \scrP^i(T_pM)\).
\end{proof}

For a real-analytic metric, the local existence of admissible polynomials
follows from:

\begin{proposition}[\cite{ZS}, Proof of Theorem~5, p.~148]\label{p:ZaSa}
The ring of locally convergent formal power series \(\bbR\{X_1,\ldots,X_n\}\)
in \(n\) variables is Noetherian.
\end{proposition}

Here, a formal power series in \(X_1,\ldots,X_n\) is called locally
convergent if its radius of convergence at the origin is positive.

\begin{corollary} \label{co:JR_loc_analytic}
Let \((M,g)\) be a real-analytic Riemannian manifold. For each
\(p \in M\) there exist an open neighborhood \(U \subseteq M\) of \(p\) and
some \(P \in \Gamma(\scrP(U))[\lambda]\) that is admissible on \(U\).
\end{corollary}

\begin{proof}
Choosing a local analytic frame field \(E_1,\ldots,E_n\) at \(p\), the ring
of germs of real-analytic sections of \(\scrP(M)\) at \(p\) is isomorphic
to the ring \(\bbR\{X_1,\ldots,X_n\}[Y_1,\ldots,Y_n]\) of polynomials in
\(Y_1,\ldots,Y_n\) with coefficients in \(\bbR\{X_1,\ldots,X_n\}\).
Proposition~\ref{p:ZaSa}, combined with Hilbert's basis theorem, implies
that this ring is Noetherian. Hence the ring of germs of analytic sections
of \(\scrP(M)\) at \(p\) is Noetherian. Let \(\{F^i\}\) be the dual basis of
\(\{E_i\}\). Then germs of real-analytic sections of
\(\scrP(M) \otimes \End{TM}\) can be uniquely written as
\(\sum_{i,j} c_{i,j} \otimes E_i \otimes F^j\) with coefficients \(c_{i,j}\)
in the ring of germs of real-analytic sections of \(\scrP(M)\) at \(p\).
Thus, the module of germs of sections of \(\scrP(M) \otimes \End{TM}\) at
\(p\) is finitely generated over the ring of germs of sections of
\(\scrP(M)\). Therefore, the same algebraic argument as in
Lemma~\ref{le:admissible} shows that there exist an open
neighborhood \(U\) of \(p\) and real-analytic sections \(a_i\) of
\(\scrP^i(U)\) for \(i=1,\ldots,k\) such that
\begin{equation} \label{eq:JR_on_U}
  \symR^k|_U = -\sum_{i=1}^k a_i\,\symR^{k-i}|_U
\end{equation}
Consequently, \(P \coloneqq \lambda^k + \sum_{i=1}^k a_i\,\lambda^{k-i}
\in \Gamma(\scrP(U))[\lambda]\) is admissible on \(U\).
\end{proof}

\paragraph{Proof of Theorem~\ref{th:globally_admissible}}
By Corollary~\ref{co:JR_loc_analytic}, there exists a finite open covering
\(\{U_j\}_{j=1,\ldots,N}\) of \(M\) such that for each \(j\) there is a
\(P_j \in \Gamma(\scrP(U_j))[\lambda]\) admissible on \(U_j\). Let \(k_j\)
denote the degree of \(P_j\), and define
\(k_{\mathrm{max}} \coloneqq \max\{k_j\}_{j=1,\ldots,N}\). By covariantly
differentiating and symmetrizing \eqref{eq:JR_on_U}, each \(P_j\) generates
an admissible polynomial of order \(k_j+1\), and this process continues
iteratively. Thus, we may assume that all \(P_j\) are monic and
homogeneous of the same degree \(k \coloneqq k_{\mathrm{max}}\):
\[
  P_j = \sum_{i=0}^k a_{i,j}\,\lambda^{k-i}  \quad\text{with}\quad
  a_{i,j} \in \Gamma(\scrP^i(U_j)), \quad a_{0,j} = 1
\]
Moreover, there exists a smooth partition \(\{\mu_j\}_{j=1,\ldots,N}\) of
unity subordinate to the open cover \(\{U_j\}\). That is, each
\(\mu_j \colon M \to [0,1]\) is smooth with compact support in \(U_j\) and
\(\sum_{j=1}^N \mu_j = 1\). Then
\begin{align*}
  \symR^k
    &= \left( \sum_{j=1}^N \mu_j \right)\,\symR^k
      = -\sum_{j=1}^N \mu_j
        \left( \sum_{i=1}^k a_{i,j}\,\symR^{k-i} \right) \\
    &= -\sum_{i=1}^k \left( \sum_{j=1}^N \mu_j\,a_{i,j} \right)
        \symR^{k-i}
      = -\sum_{i=1}^k a_i\,\symR^{k-i}
\end{align*}
with smooth coefficients \(a_i \coloneqq \sum_{j=1}^N \mu_j\,a_{i,j}
\in \Gamma(\scrP^i(M))\). Therefore,
\(P \coloneqq \lambda^k + \sum_{i=1}^k a_i\,\lambda^{k-i}
\in \Gamma(\scrP(M))[\lambda]\) is admissible. \qed

\section*{Acknowledgements}
I am grateful to Antonio Naveira, who introduced me to these ideas during
his visit to Stuttgart in 2013. I was deeply saddened to learn of his
passing in October 2021. I would also like to thank Gregor Weingart for
helpful suggestions.


\bibliographystyle{amsplain}

\end{document}